\documentclass[pdflatex,sn-mathphys-num]{sn-jnl}
\usepackage [T1]{fontenc}
\usepackage [utf8]{inputenc}
\usepackage [english]{babel}
\usepackage {amsfonts}
\usepackage {amssymb}
\usepackage {mathbbol}

\usepackage {amsopn}
\usepackage {mathtools}
\usepackage {tikz-cd}
\usepackage{csquotes}
\usepackage{xfrac}
\usepackage{faktor}
\usepackage{dsfont}
\usepackage{bm}
\usepackage{amsthm,lipsum}
\makeatletter
\def\thm@space@setup{%
  \thm@preskip=8pt plus 2pt minus 4pt
  \thm@postskip=\thm@preskip 
}
\makeatother

\usepackage{smartdiagram}

\usepackage[shortcuts]{extdash}

\newcommand{\ols}[1]{\mskip.5\thinmuskip\overline{\mskip-.5\thinmuskip {#1} \mskip-.5\thinmuskip}\mskip.5\thinmuskip} 
\newcommand{\olsi}[1]{\,\overline{\!{#1}}} 
\makeatletter
\newcommand\closure[1]{
  \tctestifnum{\count@stringtoks{#1}>1} 
  {\ols{#1}} 
  {\olsi{#1}} 
}
\long\def\count@stringtoks#1{\tc@earg\count@toks{\string#1}}
\long\def\count@toks#1{\the\numexpr-1\count@@toks#1.\tc@endcnt}
\long\def\count@@toks#1#2\tc@endcnt{+1\tc@ifempty{#2}{\relax}{\count@@toks#2\tc@endcnt}}
\def\tc@ifempty#1{\tc@testxifx{\expandafter\relax\detokenize{#1}\relax}}
\long\def\tc@earg#1#2{\expandafter#1\expandafter{#2}}
\long\def\tctestifnum#1{\tctestifcon{\ifnum#1\relax}}
\long\def\tctestifcon#1{#1\expandafter\tc@exfirst\else\expandafter\tc@exsecond\fi}
\long\def\tc@testxifx{\tc@earg\tctestifx}
\long\def\tctestifx#1{\tctestifcon{\ifx#1}}
\long\def\tc@exfirst#1#2{#1}
\long\def\tc@exsecond#1#2{#2}
\makeatother

\usepackage{yhmath}

\newcommand\blfootnote[1]{%
  \begingroup
  \renewcommand\thefootnote{}\footnote{#1}%
  \addtocounter{footnote}{-1}%
  \endgroup
}

\newcommand{\catname}[1]{{\mathbf{#1}}}
\newcommand{\C}{{\mathbb{C}_\infty}}

\newcommand{\K}{{K_\infty}}

\newcommand{\FF}{{\closure{\mathbb{F}_q}}}
\newcommand{\F}{\mathbb{F}}

\newcommand{\R}{\mathbb{R}}
\newcommand{\Z}{\mathbb{Z}}
\renewcommand{\P}{\mathcal{P}}

\renewcommand{\O}{\mathcal{O}}
\newcommand{\U}{\mathcal{U}}

\newcommand{\G}{\mathbb{G}}

\DeclareMathOperator{\Frac}{Frac}

\DeclareMathOperator{\Hom}{Hom}
\DeclareMathOperator{\res}{res}
\DeclareMathOperator{\Lie}{Lie}
\DeclareMathOperator{\Span}{Span}
\DeclareMathOperator{\End}{End}

\DeclareMathOperator{\Mat}{Mat}

\DeclareMathOperator{\Sf}{Sf}
\newcommand{\Mod}{\catname{Mod}}

\newtheorem{teo}{Theorem}[section]
\newtheorem{Teo}{Theorem}
\newtheorem*{Teo*}{Theorem}
\newtheorem{lemma}[teo]{Lemma}

\newtheorem{prop}[teo]{Proposition}
\newtheorem{Prop}[Teo]{Proposition}
\newtheorem*{prop*}{Proposition}
\newtheorem{cor}[teo]{Corollary}
\newtheorem{conj}[teo]{Conjecture}
\theoremstyle{definition}
\newtheorem{Def}[teo]{Definition}
\newtheorem{Deff}[Teo]{Definition}

\newtheorem{oss}[teo]{Remark}

\newtheorem{ex}[teo]{Example}

\begin{document}
\title{A duality result about special functions for Drinfeld modules of arbitrary rank}
\author*[]{\fnm{Giacomo Hermes} \sur{Ferraro}}\email{giacomohermes.ferraro@gmail.com}

\affil*{\orgdiv{Interdisziplin\"ares Zentrum f\"ur wissenschaftliches Rechnen}, \orgname{Ruprecht-Karls-Universit\"at Heidelberg}, \orgaddress{\street{Im Neuenheimer
 Feld 205}, \city{Heidelberg}, \postcode{69120}, \state{Baden-W\"urttemberg}, \country{Germany}}}

\abstract{
In the setting of a Drinfeld module $\phi$ over a curve $X/\F_q$, we use a functorial point of view to define \emph{Anderson eigenvectors}, a generalization of the so called "special functions" introduced by Anglès, Ngo Dac and Tavares Ribeiro, and prove the existence of a universal object $\omega_\phi$. 

We adopt an analogous approach with the adjoint Drinfeld module $\phi^*$ to define \textit{dual Anderson eigenvectors}. The universal object of this functor, denoted by $\zeta_\phi$, is a generalization of Pellarin zeta functions, can be expressed as an Eisenstein-like series over the period lattice, and its coordinates are analytic functions from $X(\C)\setminus\{\infty\}$ to $\C$.

For all integers $i$ we define dot products $\zeta_\phi\cdot\omega_\phi^{(i)}$ as certain meromorphic differential forms over $X_\C\setminus\{\infty\}$, and prove they are actually rational. This amounts to a generalization of Pellarin's identity for the Carlitz module, and is linked to the pairing of the $A$\=/motive and the dual $A$\=/motive defined by Hartl and Juschka.

Finally, we develop an algorithm to compute the forms $\zeta_\phi\cdot\omega_\phi^{(i)}$ when $X=\mathbb{P}^1$, and prove a conjecture of Gazda and Maurischat about the invertibility of special functions for Drinfeld modules of rank $1$.
}
\maketitle
\blfootnote{This version of the manuscript has been published by \href{https://www.doi.org/10.1007/s40687-025-00506-w}{Res. Math. Sci.} under the licensing agreement \href{https://creativecommons.org/licenses/by/4.0/}{CC BY 4.0}.}

\section{Introduction}

Drinfeld modules are meant to provide an analogue of complex elliptic curves---interpreted as quotients of the complex plane by a lattice---in the context of function fields over a finite field $\F_q$. Instead of $\mathbb{Z}$, we work with the ring $A$ of functions over a projective, geometrically irreducible, smooth curve ${X}/{\F_q}$ outside a closed point $\infty$. The role of the real numbers is filled by the $\infty$\=/adic completion $\K$ of the function field of $X$, while the field of complex numbers is substituted by $\C$, defined as the $\infty$\=/adic completion of an algebraic closure of $\K$.

Inside $\C$, we can consider a discrete projective finitely generated $A$\=/module $\Lambda$, the \emph{period lattice}, which contrary to the theory of elliptic curves can have arbitrary rank $r$, and is never cocompact. A Drinfeld module $\phi$ is an $\F_q$\=/linear and polynomial action of $A$ on $\C$, and Drinfeld modules are in bijection with period lattices in the following way: for each Drinfeld module $\phi$ there is a unique period lattice $\Lambda_\phi$ such that $\C$, endowed with the $A$\=/module structure induced by $\phi$, is isomorphic to ${\C}/{\Lambda_\phi}$ with the $A$\=/module structure induced by the inclusion $A\subseteq\C$.

\subsubsection*{Entire functions.}

The Tate algebra $\C\closure\otimes A$ is defined as the completion of $A_\C\coloneqq\C\otimes A$ with respect to the sup norm induced by $\C$ (when not specified, tensor products are assumed to be over $\F_q$); given an $\F_q$\=/basis $\{a_i\}_i$ of $A$, all elements of $\C\closure\otimes A$ can be uniquely expressed as $\sum_i c_i\otimes a_i$ with $\lim_i c_i=0$.

The Tate algebra $\C\closure\otimes A$ can be thought of as the set of analytic functions from the "unit disc" $D\subseteq X(\C)$ to $\C$, where \[D\coloneqq\left\{P\in X(\C)\setminus\{\infty\}\text{ such that }\|a(P)\|\leq1\text{ for all }a\in A\right\}.\]

In the article \cite{CNDP21}, Chung, Ngo Dac, and Pellarin proved that, if $\infty\in X(\F_q)$, the \emph{Pellarin zeta function} \[\zeta_A\coloneqq-\sum_{a\in A\setminus\{0\}}a^{-1}\otimes a\in\C\closure\otimes A\] is an analytic function from $X(\C)\setminus\{\infty\}$ to $\C$.
We generalize this result as follows to a wider class of Eisenstein-like series.

\begin{Prop}[Prop. \ref{prop analyticity of zeta}]
    Let $\Lambda\subseteq\C$ be an arbitrary period lattice, and define \[\zeta_\Lambda\coloneqq-\sum_{\lambda\in\Lambda\setminus\{0\}}\lambda^{-1}\otimes\lambda\in\C\closure\otimes\Lambda.\] For any $A$\=/linear map $f:\Lambda\to A$, the element $(1\otimes f)\zeta_\Lambda\in\C\closure\otimes A$ is an analytic function from $X(\C)\setminus\{\infty\}$ to $\C$.
\end{Prop}

The most interesting aspect of this proposition is that it is proven in a completely different way from \cite{CNDP21}: it is a simple consequence of one of the main theorems of this paper, namely that $\zeta_\Lambda$ is an "eigenvector" for the adjoint Drinfeld module associated to $\Lambda$ (see Theorem \ref{thm 4}). This property of $\zeta_\Lambda$ is the motivating result of this paper, and is meant to mirror the property of Anderson--Thakur special functions.

\subsubsection*{Special functions.}

The simplest example of a Drinfeld module is the \textit{Carlitz module} $C$: we assume ${X=\mathbb{P}^1_{\F_q}}$, so that $A=\F_q[\theta]$ for some rational function $\theta$, and set $\phi_\theta\coloneqq\theta+\tau$, where $\tau:\C\to\C$ denotes the Frobenius endomorphism sending $c$ to $c^q$. In this case, the period lattice is $\tilde{\pi}A$ for some $\tilde{\pi}\in\mathbb{C}_\infty^\times$; as we said before, we can identify the Tate algebra $\C\closure\otimes A$ with the set of analytic functions from the unit disc of $\C$ to $\C$, i.e. with the set of formal series $\sum_i s_it^i\in\C[\![t]\!]$ such that $\lim_i s_i=0$.

Anderson and Thakur introduced in \cite{AT90} the function $\omega\in\C\closure\otimes A$ as the unique element such that, if we write $\omega=\sum_{i\geq0}c_it^i$, we have $c_0=1$ and
\[\sum_{i\geq0}\phi_\theta(c_i)t^i=\sum_{i\geq0}c_i t^{i+1}.\]
This series has various uses: for example, as shown in \cite{AP14} by Anglès and Pellarin, $\omega$ is connected to the explicit class field theory of $\F_q(\theta)$, and its $\FF$\=/rational values interpolate Gauss-Thakur sums. 

The module of "special functions" (as defined in \cite{ANDTR17a} by Anglès, Ngo Dac, and Tavares Ribeiro) generalizes the Anderson--Thakur function to any Drinfeld module $\phi$ as follows:
\[\Sf_\phi(A)\coloneqq\{\omega\in\C\closure\otimes A\mid(\phi_a\otimes1)(\omega)=(1\otimes a)\omega\text{ for all } a\in A\},\]
where $\phi_a\otimes 1$ sends an infinite series $\sum_i c_i\otimes a_i\in\C\closure\otimes A$ to $\sum_i\phi_a(c_i)\otimes a_i$.

In a recent article, Gazda and Maurischat showed (in the generality of an arbitrary anderson module $\underline{E}=(E,\phi)$) that the module of special functions is isomorphic to $\Hom_A(\Omega,\Lambda_\phi)$, where $\Omega$ is the module of K\"ahler differentials of $A$ (\cite[Thm. 3.11]{GM21}). 

In this paper we formulate a generalization of special functions which allows us to recover this result using the language of functors. We give the following Definition and Theorem, in the generality of Anderson modules (see Section \ref{section special functions} for definitions).
\begin{Deff}[Def. \ref{Def special functor}]
    For any Anderson $A$\=/module $(E,\phi)$, the functor of \textit{Anderson eigenvectors} \[\Sf_\phi:A\mbox{-}\Mod\to A\mbox{-}\Mod\] sends a discrete $A$\=/module $M$ to the $A$\=/module:
    \[\Sf_\phi(M)\coloneqq\{\omega\in E(\C)\closure\otimes M\mid(\phi_a\otimes1)(\omega)=(1\otimes a)\omega\text{ for all }a\in A\}.\]
\end{Deff}

\begin{Teo}[Thm. \ref{main result},\ref{non uniformizable case}]
    Let $(E,\phi)$ be an Anderson $A$\=/module. If either $(E,\phi)$ is uniformizable or we restrict the functor $\Sf_\phi$ to the category of torsionfree $A$\=/modules, $\Sf_\phi$ is represented by the $A$\=/module $\Hom_A(\Lambda_\phi,\Omega)$.
\end{Teo}
This representability result is reminiscent of the commutative diagram in \cite[Thm. 5.2]{GM21}, involving the module of special functions and the module of Gauss--Thakur sums for a given character $\chi:A\to\FF$.

While there is no canonical special function for arbitrary Anderson modules, there is a canonical Anderson eigenvector, namely the universal object of $\Sf_\phi$:
\[{\omega_\phi\in E(\C)\closure\otimes\Hom_A(\Lambda_\phi,\Omega)}.\] Moreover, given an $\F_q$\=/basis of $\Hom_A(\Lambda_\phi,\Omega)$, it's possible to write an explicit series expansion of $\omega_\phi$ in terms of the exponential map $\exp_E$ (see Remark \ref{oss universal special function}).
Using this expansion, we are able to answer positively a conjecture by Gazda and Maurischat from the article \cite{GM21} as follows.
\begin{Teo}[Thm. \ref{teo Gazda}]
    Assume that $\Sf_\phi(A)$ is free of rank $1$. Then, there is a special function in $\Sf_\phi(A)$ which is invertible as an element of $\C\closure\otimes A$.
\end{Teo}

\subsubsection*{Dual special functions.}

Given a Drinfeld $A$\=/module $(\G_a,\phi)$, it's possible to induce a natural $\F_q$\=/linear action $\phi^*$ of $A$ on $\C$, called the \emph{adjoint Drinfeld module} (see Section \ref{section zeta functions} for details). 

In the paper \cite{Ferraro}, the author proved that, assuming $\infty\in X(\F_q)$, the following holds for any Drinfeld module $\phi$ of rank $1$.

\begin{Teo*}[{\cite[Thm. 7.23]{Ferraro}}]
    Let $\tilde{\pi}I$ be the period lattice associated to $\phi$, where $I\subseteq A$ is an appropriate nonzero ideal of $A$, and define \[\zeta_I\coloneqq-\sum_{a\in I\setminus\{0\}}a^{-1}\otimes a\in\C\closure\otimes A.\] Then, the following identity holds in $\C\closure\otimes A$ for all $a\in A$:
    \[(\phi^*_a\otimes1)\left((\tilde{\pi}^{-1}\otimes1)\zeta_I\right)=(1\otimes a)(\tilde{\pi}^{-1}\otimes1)\zeta_I.\]
\end{Teo*}

This result prompts the definition of \emph{dual Anderson eigenvectors}, mirroring Definition 1.

\begin{Deff}[Def. \ref{Def dual Anderson eigenvectors}]
    For any Drinfeld $A$\=/module $(\G_a,\phi)$, the functor of \textit{dual Anderson eigenvectors} \[\Sf_{\phi^*}:A\mbox{-}\Mod\to A\mbox{-}\Mod\] sends a discrete $A$\=/module $M$ to the $A$\=/module:
    \[\Sf_{\phi^*}(M)\coloneqq\{\omega\in \C\closure\otimes M\mid(\phi^*_a\otimes1)(\omega)=(1\otimes a)\omega\text{ for all }a\in A\}.\]
\end{Deff}
We prove the following:
\begin{Teo}[Thm. \ref{zeta function functor}]\label{thm 4}
The functor $\Sf_{\phi^*}$ is represented by the $A$\=/module $\Lambda_\phi$, and
the universal object is:
\[\zeta_\phi\coloneqq-\sum_{\lambda\in\Lambda_\phi\setminus\{0\}}\lambda^{-1}\otimes\lambda\in\C\closure{\otimes}\Lambda_\phi.\]
\end{Teo}

\subsubsection*{A generalization of Pellarin's identity}

In \cite{Pellarin12}, Pellarin proved the following identity in $\C\closure\otimes A$ when $A=\F_q[\theta]$:
\begin{equation}\label{eq. Pellarin}
    \frac{\tilde{\pi}}{(\theta\otimes1-1\otimes\theta)\omega_C}=\zeta_A,
\end{equation}
where $\omega_C$ is the special function of Anderson--Thakur. and $\tilde{\pi}A\subseteq\C$ is the period lattice of the Carlitz module.

In a letter to the author in 2021, Pellarin conjectured that the product of $\zeta_A$ and a special function belongs to the fraction field of $A_\C$ for any $A$ and any Drinfeld module of rank $1$: Green and Papanikolas had already proven this statement when $X$ is an elliptic curve (\cite[Thm. 7.1]{GP18}). 

In the paper \cite{Ferraro}, the author proved this conjecture for $X$ of any genus assuming $\infty\in X(\F_q)$ ({\cite[Thm. 6.3]{Ferraro}}). In this paper, we generalize the result \cite[Thm. 6.3]{Ferraro} as follows, where we set $\omega^{(k)}\coloneqq (\tau^k\otimes1)(\omega_\phi)$ for any integer $k$ and we denote by $\Omega_\C$ the module of K\"ahler differentials of $A_\C$ as a $\C$\=/algebra, i.e. $\C\otimes_A\Omega$.

\begin{Teo}[Thm. \ref{Teo rationality}]
    Let $\phi$ be an arbitrary Drinfeld module, and denote by a dot product the natural $\C\closure\otimes A$\=/bilinear map
    \[\cdot:\C\closure\otimes\Lambda_\phi\times\C\closure\otimes\Hom_A(\Lambda_\phi,\Omega)\to\C\closure\otimes\Omega.\]
    
    For all integers $k$, $\zeta_\phi\cdot\omega^{(k)}$ is a rational differential form over the base-changed curve $X_{\C}$. Moreover, for all positive integers $k$, $\zeta_\phi\cdot\omega^{(k)}\in \Omega_\C$.
\end{Teo}

In this paper, working with a Drinfeld module $\phi$ of arbitrary rank $r$, we lose the knowledge of the divisors of the differential forms $\zeta_\phi\cdot\omega_\phi^{(k)}$ for a generic curve, which are instead explicitly described in \cite{Ferraro}. On the other hand, we are able to prove the following result about their generating series, where we identify $\C\closure\otimes\Omega$ with the set of continuous $\F_q$\=/linear homomorphisms from ${\K}/{A}$ to $\C$, as per Proposition \ref{complete tensor product} and \cite[Thm. 8]{Poonen96}.

\begin{Teo}[Thm. \ref{main teo}]
Let $\Phi,\hat{\Phi}:\K\to\C[\![\tau]\!][\tau^{-1}]$ be the unique ring homomorphisms which extend respectively $\phi,\phi^*:A\to\C[\![\tau]\!][\tau^{-1}]$ and such that their $k$\=/th coefficient is a continuous function from $\K$ to $\C$ for all $k\in\Z$. 

The following identity holds in the $\C[\tau,\tau^{-1}]$\=/module $\C[\![\tau,\tau^{-1}]\!]$ for all $c\in\K$:
    \[\sum_{k\in\mathbb\Z} \left(\zeta_\phi\cdot\omega_\phi^{(k)}\right)(c)\tau^k=\Phi_c^*-\hat{\Phi}_c.\]
\end{Teo}

\subsubsection*{Some explicit computations.}

In the case $A=\F_q[\theta]$, so that $\Omega=Ad\theta$, the previous theorem allows us to prove the following result.
\begin{Teo}[Thm. \ref{g=0 i=0}]
    Assume $A=\F_q[\theta]$ and let $\phi$ be a Drinfeld module of rank $r$. We have the following identities in $\C\closure\otimes\Omega$:
    \begin{align*}
        \zeta_\phi\cdot\omega_\phi&=\frac{d\theta}{\theta\otimes1-1\otimes\theta};\\
        \zeta_\phi\cdot\omega_\phi^{(k)}&=0\text{ if }1\leq k\leq r-1.
    \end{align*}
\end{Teo}
The previous identities, when $r=1$, imply the original identity (\ref{eq. Pellarin}) proved by Pellarin in \cite{Pellarin12} in the context of the Carlitz module.
Moreover, knowing the coefficients of $\phi_\theta$, we can compute recursively $\zeta_\phi\cdot\omega_\phi^{(k)}$ for all $k$ using the functional identity of $\omega_\phi$. 

Theorem \ref{main teo} also allows us to outline an algorithm to compute the differential forms for any given Drinfeld module on any given curve.
As an example, we apply this algorithm to the simple case of a normalized Drinfeld module $\phi$ of rank $1$ on a hyperelliptic curve $X$ of genus $g\geq1$, so that $\Omega=A\nu$ for a certain $\nu\in\Omega$, to recover an explicit formula generalizing the results originally found by Green and Papanikolas in \cite{GP18}.
\begin{Teo}[Thm. \ref{prop hyper}]
    Assume that $A={\F_q[x,y]}/{y^2-Q(x)y-P(x)}$, with $\deg(P)=2g+1$ and $\deg(Q)\leq g$.
    We have the following identities for the dot product $\zeta_\phi\cdot\omega_\phi$ and the shtuka function $f_\phi$:
    \begin{align*}
    \zeta_\phi\cdot\omega_\phi&=\left(\frac{y\otimes 1+1\otimes (y-Q(x))}{x\otimes1-1\otimes x}-\sum_{i=0}^{g-1}\left(\phi^*_{yx^{-i-1}}\right)_0\otimes x^i\right)(1\otimes \nu)\\
        f_\phi&=\frac{(x\otimes1-1\otimes x)\left(-\sum_{i=0}^g \left(\phi^*_{yx^{-i-1}}\right)_1\otimes x^i\right)}{y\otimes 1+1\otimes (y-Q(x))-(x\otimes1-1\otimes x)\left(\sum_{i=0}^{g-1}\left(\phi^*_{yx^{-i-1}}\right)_0\otimes x^i\right)}.
    \end{align*}
\end{Teo}

Given a Drinfeld module $\phi$, we denote by $M(\phi)$ its $A$\=/motive, and by $N(\phi)$ its dual $A$\=/motive (see Definition \ref{def (dual) motive}). Hartl and Juschka proved in \cite[Thm. 5.13]{HJ20}---in the wider generality of abelian and $A$\=/finite Anderson modules---that there is an isomorphism of left $A_\C[\tau^{-1}]$\=/modules between $N(\phi)$ and $\Hom_{A_\C}(\tau M(\phi),\Omega_\C)$. Using the explicit computations of Theorem \ref{g=0 i=0}, we prove the following:
\begin{Teo}[Prop. \ref{prop motive embedding}, Thm. \ref{teo restricted pairing}]
    Let $\phi$ be a Drinfeld module. There are a canonical embedding of left $A_\C[\tau]$\=/modules $\tau M(\phi)\subseteq\C\overline\otimes\Hom_A(\Lambda_\phi,\Omega)$ and a canonical embedding of left $A_\C[\tau^{-1}]$\=/modules $N(\phi)\subseteq\C\overline\otimes\Lambda_\phi$. 
    
    If we assume $A=\F_q[\theta]$, the restriction of the dot product to these submodules coincides with the perfect pairing induced by Hartl--Juschka's isomorphism.
\end{Teo}

\subsection*{Acknowledgements}
The author is thankful to Matt Papanikolas for the interesting and enlightening discussions relating to the contents of the present paper, and to Federico Pellarin for his guidance and suggestions.
The author thanks the Department of Mathematics Guido Castelnuovo, Università La Sapienza, where he has carried out his Ph.D. studies. This work has been produced as part of the Ph.D. thesis of the author.



\section{Pontryagin duality of \texorpdfstring{$A$}{A}-modules}\label{section Pontryagin duality}

\subsection{Basic statements about Pontryagin duality}

In this paper, compact and locally compact spaces are always assumed to be Hausdorff.

\begin{Def}[Pontryagin duality]
    Call $\mathbb{S}^1\subseteq\mathbb{C}^\times$ the complex unit circle. For any commutative ring with unity $B$, the \textit{Pontryagin duality} is a contravariant functor from the category of topological $B$\=/modules to itself, sending a module $M$ to the set of continuous group homomorphism $\hat{M}\coloneqq\Hom_\Z^\mathrm{cont}(M,\mathbb{S}^1)$, endowed with the compact open topology and with the natural $B$\=/module structure.
\end{Def}
The following well known result, which we do not prove, justifies the terminology "duality".
\begin{prop}\label{Pontryagin duality}
    For any ring $B$ and any topological $B$\=/module $M$, consider the group homomorphism $i_M:M\to\hat{\hat M}$ sending $m\in M$ to $(f\mapsto f(m))$. The map $i_M$ is a continuous $B$\=/linear homomorphism; if $M$ is locally compact, $\hat{M}$ is locally compact, and $i_M$ is an isomorphism. Moreover, if $M$ is compact (resp. discrete) $\hat{M}$ is discrete (resp. compact).
\end{prop}

Let now $X$ be a projective, geometrically irreducible, smooth curve over $\F_q$, with a closed point $\infty\in X$ of degree $e$, and let $A$ be the ring of rational functions over $X$ with no poles outside $\infty$.

\begin{oss}
    If $M$ is an $A$\=/module, since $M$ is also an $\F_q$\=/vector space, we have the following natural isomorphisms of topological $A$\=/modules:
    \[\hat{M}\coloneqq\Hom^\mathrm{cont}_\Z(M,\mathbb{S}^1)\cong\Hom^\mathrm{cont}_{\F_q}(M,\Hom_{\Z}(\F_q,\mathbb{S}^1))=\Hom^\mathrm{cont}_{\F_q}(M,\hat\F_q).\]
\end{oss}

Fix an isomorphism of the $p$-torsion points of $\mathbb{S}^1$ with $\F_p$, where $p$ denotes the characteristic of $\F_q$; we can identify the $\F_q$-vector spaces $\F_q$ and $\hat{\F}_q$ by sending $1$ to the trace map $\mathrm{tr}_{\F_q/\F_p}:\F_q\to\F_p$ so that, from now on, we can write $\hat{M}=\Hom_{\F_q}^\mathrm{cont}(M,\F_q)$ for any $\F_q$\=/vector space $M$. Let's fix some additional notation.

\begin{Def}\label{complete tensor product}
Let $M$ and $N$ be topological $\F_q$\=/vector spaces with $N$ locally compact. We define the topological tensor product of $M$ and $N$, and denote by $M\hat\otimes N$, the space $\Hom_{\F_q}^\mathrm{cont}(\hat{N},M)$ of continuous $\F_q$\=/linear homomorphisms from $\hat{N}$ to $M$.
\end{Def}

\begin{oss}
    The topological tensor product can be endowed with the compact open topology, but we will only need to use the definition of the underlying set.
\end{oss}

\begin{lemma}\label{duality tensor hom 1}
    For any pair of locally compact $A$\=/modules $M,N$, there is a natural isomorphism of $A\otimes A$\=/modules between $M\hat\otimes N$ and $N\hat\otimes M$.
\end{lemma}
\begin{proof}
    By Proposition \ref{Pontryagin duality}, the Pontryagin duality induces an antiequivalence of the category of locally compact $\F_q$\=/vector spaces with itself, hence we have the following natural bijections:
    \[\Hom_{\F_q}^\mathrm{cont}(\hat N,M)\cong\Hom_{\F_q}^\mathrm{cont}(\hat{M},\hat{\hat N})\cong\Hom_{\F_q}^\mathrm{cont}(\hat M,N);\]
    the $A\otimes A$\=/linearity is a simple check.
\end{proof}

We introduce some other useful terminology.

\begin{oss}
    For any set $I$, the Pontryagin dual of $\F_q^{\oplus I}$ can be identified with $\F_q^I$. In particular, for any discrete $\F_q$\=/vector space $M$, an isomorphism $\F_q^{\oplus I}\cong M$, i.e. an $\F_q$\=/basis $(m_i)_{i\in I}$, induces an isomorphism of topological vector spaces between $\F_q^I=\widehat{\F_q^{\oplus{I}}}$ and $\hat{M}$.
\end{oss}
\begin{Def}
    If $M$ is a discrete $\F_q$\=/vector space with basis $(m_i)_{i\in I}$, for all $i\in I$ we denote by $m_i^*$ the image of $(\delta_{i,j})_{j\in I}\in \F_q^I$ via the isomorphism with $\hat{M}$, so that for all $j\in I$ $m_i^*(m_j)=\delta_{i,j}$. We call $(m_i^*)_{i\in I}$ the \textit{dual basis} of $\hat{M}$ relative to $(m_i)_{i\in I}$.
\end{Def}

\begin{oss}
    In the previous definition, a generic element $f\in\hat{M}$ corresponds to $(f(m_i))_i\in\F_q^I$. It's immediate to check that, for all $m\in M$, \[f(m)=\sum_{i\in I}f(m_i)m_i^*(m),\] which is actually a finite sum, hence we are justified in the use the following formal notation: \[f=\sum_{i\in I}f(m_i)m_i^*.\] The existence and uniqueness of this expression for all $f\in\hat{M}$ explains the terminology "dual basis" for $(m_i^*)_i$.
\end{oss}

\subsection{Application to \texorpdfstring{$A$}{A}-modules}

Denote by $\K$ the completion of the fraction field $K$ of $A$ at $\infty$, where the norm of an element $c\in K$ is defined as $q^{-v_\infty(c)}$, and denote by $\C$ the completion of an algebraic closure of $\K$. Denote by $\Omega$ the module of K\"ahler differentials of $A$, which is a projective $A$\=/module of rank $1$.

The following is a fundamental result about the Pontryagin duality of $A$\=/modules, known as \emph{residue duality} (see \cite[Thm. 8]{Poonen96}).

\begin{teo}
    The computation of the residue at $\infty$ induces a perfect pairing between $\Omega\otimes_A\K$ and $\K$, which restricts to a perfect pairing between the discrete $A$\=/module $\Omega$ and the compact $A$\=/module ${\K}/{A}$. In other words, $\widehat{\Omega\otimes_A\K}\cong\K$ and $\hat\Omega\cong{\K}/{A}$.   
\end{teo}

\begin{oss}\label{oss dual}
    For any discrete projective $A$\=/module $\Lambda$ of finite rank $r$, we have the following natural isomorphisms of topological $A$\=/modules, where $\Lambda^*\coloneqq\Hom_A(\Lambda,A)$:
    \[\widehat{\Lambda^*\otimes_A\Omega}=\Hom_{\F_q}(\Lambda^*\otimes_A\Omega,\F_q)\cong\Hom_A(\Lambda^*,\Hom_{\F_q}(\Omega,\F_q))\cong\Lambda\otimes_A\left(\faktor{\K}{A}\right).\]
    Retracing the isomorphisms, it's easy to check that the pairing \[(\Lambda^*\otimes_A\Omega)\otimes\left(\Lambda\otimes_A{\K}/{A}\right)\to \F_q\] sends the element $(\lambda^*\otimes\omega)\otimes(\lambda\otimes b)$ to the pairing of $\lambda^*(\lambda)b\otimes\omega\in{\K}/{A}\otimes_A\Omega$.
\end{oss}

We now show that in some cases the topological tensor product of two spaces is naturally isomorphic to a completion of their tensor product. This makes our notation agree with the usual notation $\C\hat\otimes A$ employed for the Tate algebra in works like \cite{GM21}, \cite{Ferraro}, and others.

\begin{Def}\label{Def complete space}
    Let $C$ be a topological vector space which is the projective limit of a diagram of discrete $\F_q$\=/vector spaces $\{C_i\}_{i\in I}$: we call such a space a \textit{prodiscrete} $\F_q$\=/vector space; we call its \textit{associated filter} the collection $\U\coloneqq\{\ker(C\to C_i)\}_{i\in I}$, which is a neighborhood filter of $0$ comprised of open (and closed) subspaces of $C$. 
    
    For any discrete $\F_q$\=/vector space $M$ and any prodiscrete $\F_q$\=/vector space $C$, we denote by $C\closure\otimes M$ the completion of $C\otimes M$ with respect to the neighborhood filter of $0$ given by $\{U\otimes M\}_{U\in\U}$.
\end{Def}

\begin{ex}
    The open ball $B_r\subseteq\C$ of radius $r\in\R_{>0}$ is an $\F_q$\=/vector space, because the norm on $\C$ is non-archimedean. Since $\C$ is complete, $\C$ is a prodiscrete $\F_q$\=/vector space, with associated filter $\{B_r\}_{r\in\R_{>0}}$.
\end{ex}

\begin{prop}\label{duality tensor hom 2}
    Let $C$ be a prodiscrete $\F_q$\=/vector space and $M$ be a discrete $\F_q$\=/vector space. There is a natural $\F_q$\=/linear bijection $\Phi:C\closure\otimes M\to C\hat\otimes M$. 
    
    If we fix an $\F_q$\=/basis $(m_i)_{i\in I}$ of $M$ with corresponding dual basis $(m_i^*)_{i\in I}$ of $\hat{M}$, for any function $f\in C\hat\otimes M=\Hom^\mathrm{cont}_{\F_q}(\hat{M},C)$ we have \[\Phi^{-1}(f)=\sum_{i\in I} f(m_i^*)\otimes m_i.\]

    Moreover, if $C$ and $M$ are $A$\=/modules, $\Phi$ is $A\otimes A$\=/linear.
\end{prop}
\begin{proof}
    Fix an $\F_q$\=/basis $(m_i)_{i\in I}$ of $M$ and let $\U$ be an associated filter of $C$. Any $x\in C\closure\otimes M$ can be expressed in a unique way as $\sum_{i\in I}x_i\otimes m_i$, where $x_i\in C$ for all $i\in I$, and for all $U\in \U$ the set \[I_U\coloneqq\{i\in I\mid x_i\not\in U\}\] is finite. We define $\Phi(x):\hat{M}\to C$ as follows:
    \[\forall f\in\hat{M},\;\Phi(x)(f)\coloneqq\lim_{J\in\P^\mathrm{fin}(I)}\sum_{i\in J} f(m_i)x_i.\]
    Since $C$ is complete with respect to the neighborhood filter $\U$, and for all $U\in\U$ the set \[\{i\in I\mid f(m_i)x_i\not\in U\}\subseteq I_U\] is finite, the map $\Phi(x)$ is well defined.
    For all $U\in\U$, the set \[\{f\in\hat{M}\mid f(m_i)=0\text{ for all }i\in I_U\}\] is a neighborhood of $0$ in $\hat{M}$, and is contained in $\Phi(x)^{-1}(U)$, hence $\Phi(x)$ is continuous. Since $\Phi(x)$ is also obviously $\F_q$\=/linear, $\Phi(x)\in C\hat\otimes M$ for all $x\in C\closure\otimes M$.
    
    The map $\Phi$ is manifestly $\F_q$\=/linear, and, if $C$ and $M$ are $A$\=/modules, $A\otimes A$\=/linear, so we just need to prove bijectivity.
    On one hand, if $\Phi(x)\equiv0$, we have \[0=\Phi(x)(m_i^*)=x_i\text{ for all } i\in I,\] hence $x=0$, showing injectivity. On the other hand, if $g:\hat{M}\to C$ is a continuous function, for all $U\in\U$ the set \[\{i\in I\mid g(m_i^*)\not\in U\}\] is finite because $\hat{M}$ is compact, hence $y\coloneqq\sum_i g(m_i^*)\otimes m_i$ is an element of $C\closure\otimes M$; since \[\Phi(y)(m_i^*)=g(m_i^*)\text{ for all }i\in I,\] we have $\Phi(y)=g$.
\end{proof}

\section{Universal Anderson eigenvector}\label{section special functions}

In this section, we will define the functor of \textit{Anderson eigenvectors} relative to an Anderson module $(E,\phi)$, which generalizes the concept of special functions and Gauss-Thakur sums (see Definition \ref{Def special functor}), and prove that under some conditions it is representable (see Theorem \ref{main result} and Theorem \ref{non uniformizable case}). As a corollary, we get a variant of the result \cite[Thm. 3.11]{GM21}, in which Gazda and Maurischat described the module of special functions for any Anderson module $(E,\phi)$.

\subsection{Anderson modules}

\begin{Def}
    Given an $\F_q$\=/algebra $R$, an \emph{$R$\=/module scheme over $\C$} $G$ is a group scheme over $\C$ endowed with a compatible action of $R$, i.e. a ring homomorphism $R\to\End(G)$.
\end{Def}

If $G$ is a group scheme over $\C$, we denote by $\Lie(G)$ its tangent space at the identity, which has a natural structure of $\C$\=/vector space (it also has a Lie algebra structure, but it is trivial in all group schemes of interest to us). This association can be extended to a functor from the category of group schemes over $\C$ to that of $\C$\=/vector spaces, and given $f:G\to G'$ a morphism in the first category, we denote the induced morphism $\Lie(G)\to\Lie(G')$ as $\Lie(f)$.

Let's define Anderson $A$\=/modules (see \cite[Def. 2.5.2]{HJ20}).

\begin{Def}
    An \textit{Anderson $A$\=/module} $\underline{E}=(E,\phi)$ over $\C$ of dimension $d$ consists of an $A$\=/module scheme $E$ over $\C$ with the following properties:
    \begin{itemize}
        \item as an $\F_q$\=/module scheme over $\C$, $E$ is isomorphic to $\mathbb{G}_{a,\C}^d$;
        \item the action $\phi$ of $A$ on $E$ is such that $\Lie\phi_a-a:\Lie(E)\to\Lie(E)$ is nilpotent for all $a\in A$.
    \end{itemize} 
\end{Def}

Fix an Anderson $A$\=/module $(E,\phi)$.
There is a unique $\F_q$\=/linear function \[{\exp_\phi:\Lie(E)\to E(\C)},\] called \textit{exponential} of $\phi$, such that \[\exp_\phi\circ\Lie\phi_a=\phi_a\circ\exp_\phi\text{ for all }a\in A\] (see \cite[Thm. 5.9.6]{Goss98}); its kernel \[\Lambda_\phi\coloneqq\ker(\exp_\phi)\subseteq\Lie(E)\] is an $A$\=/module of finite rank with respect to the $A$\=/module structure induced by $\Lie\phi$ on $\Lie(E)$ (see \cite[Lemma 5.9.12]{Goss98}).

Moreover, if we fix an isomorphism $E\cong\mathbb{G}_{a,\C}^d$, the exponential function can be identified with a series in $\mathbb{C}_\infty^{d\times d}[\![\tau]\!]$---where $\tau$ is the Frobenius endomorphism---whose leading term is the identity matrix.

\begin{oss}\label{oss topology}
    Since $E$ and $\mathbb{G}_{a,\C}^d$ are isomorphic group schemes over $\C$, we can identify the set $E(\C)$ with $\mathbb{G}_{a,\C}^d(\C)=\mathbb{C}_\infty^d$ up to an automorphism of $\mathbb{G}_{a,\C}^d$ as a group scheme over $\C$. Since any such automorphism is continuous as a map from $\mathbb{C}_\infty^d$ to itself, $E(\C)$ has a natural topology, and is homeomorphic to $\mathbb{C}_\infty^d$.

    Since $\Lie(E)$ is a $d$\=/dimensional vector space over $\C$, it also has a natural topology; moreover, by the inverse function theorem applied to $\exp_\phi$, we get that $\Lambda_\phi\subseteq\Lie(E)$ is a discrete subset. In light of this remark, and since for all $a\in A$ ${\exp_\phi\circ\Lie\phi_a=\phi_a\circ\exp_\phi}$, $\exp_\phi$ is a morphism of topological $A$\=/modules.
\end{oss}

\begin{Def}
    Let $\underline{E}=(E,\phi)$ be an Anderson $A$\=/module. The discrete $A$\=/module $\Lambda_\phi\subseteq\Lie(E)$ is called the \emph{period lattice} of $\phi$. If $\exp_\phi$ is surjective, $\underline{E}$ is said to be \textit{uniformizable}; in this case, its \emph{rank} is defined as the rank of $\Lambda_\phi$ as an $A$\=/module.
\end{Def}

The following is a well-known lemma, so we just give an outline of the proof.

\begin{lemma}\label{K-action on Lie(E)}
    The $A$\=/module structure of $\Lie(E)$ induced by $\phi$ extends uniquely to a structure of topological vector space over $\K$.
\end{lemma}
\begin{proof}
    Since the endomorphisms $(\Lie\phi_a)_{a\in A\setminus\{0\}}$ commute and are invertible, the ring homomorphism $\Psi:A\to \End_\C(\Lie(E))$ sending $a$ to $\Lie\phi_a$ can be extended uniquely to $K$, and we can fix a basis $\Lie(E)\cong\mathbb{C}_\infty^d$ in which, for all $c\in K$, $\Psi_c$ is a triangular matrix with $N_c\coloneqq c^{-1}\Psi_c-\mathrm{Id}_d$ nilpotent---precisely, $N_c^d=0$. We can endow $\End_\C(\mathbb{C}_\infty^d)$ with a matrix norm $|\cdot|$ sending a matrix to the maximum of the norms of its coefficients; since the norm on $\C$ is non-archimedean, $|\cdot|$ is submultiplicative. To extend continuously $\Psi$ to $\K$, it suffices to prove that the set $\{|c^{-1}\Psi_c|\}_{c\in K\setminus\{0\}}$ is bounded, so that $|\Psi_c|$ tends to $0$ as $\|c\|$ tends to $0$.

    Since $A$ is a finitely generated $\F_q$\=/algebra, we can pick a finite set $\{a_1,\dots,a_n\}$ such that the finite products of the $a_i$'s generate $A$ as an $\F_q$\=/vector space. 
    
    Call $M\coloneqq\max\{1,|N_{a_1}|,\dots,|N_{a_n}|\}$. For $i=1,\dots,n$, for all $m\geq0$, since $N_{a_i}^d=0$:
    \[a_i^{-m}\Psi_{a_i^m}=(a_i^{-1}\Psi_{a_i})^m\in\Span_{\F_q}\left(\{\mathrm{Id}_d,N_{a_i},\dots,N_{a_i}^{d-1}\}\right),\] hence $|\Psi_{a_i^m}|< \|a_i\|^m M^d$. Fix any $b\in A$: since it is a linear combination of products of powers of the $a_i$'s, and since the norm $|\cdot|$ is non-archimedean, we have $|b^{-1}\Psi_b|<M^{nd}$, hence \[|N_b|\leq\max(\{1,|b^{-1}\Psi_b|\})< M^{nd}.\]
    For all $c\in K^\times$, if we write $c=ab^{-1}$ with $a,b\in A\setminus\{0\}$, we have:
    \[|c^{-1}\Psi_c|=|a^{-1}\Psi_a(b^{-1}\Psi_b)^{-1}|=\left|(\mathrm{Id}_d+N_a)\left(\sum_{i=0}^{d-1}(-N_b)^i\right)\right|< M^{nd^2},\]
    which concludes the proof.
\end{proof}

\subsection{Functor of Anderson eigenvectors}

Let's endow $E(\C)$ with the $A$\=/module structure induced by $\phi$ and with the natural topology of Remark \ref{oss topology}.

\begin{Def}\label{Def special functor}
    For any discrete $A$\=/module $M$, its set of \textit{Anderson eigenvectors} is defined as the $A$\=/module of continuous $A$\=/linear homomorphisms \[\Hom_A^\mathrm{cont}(\hat{M},E(\C))\subseteq E(\C)\hat\otimes M.\]
    We denote by \[\Sf_\phi:A\mbox{-}\Mod\to A\mbox{-}\Mod\] the functor that extends this map in the straightforward way.
\end{Def}

\begin{oss}
    By Proposition \ref{duality tensor hom 2}, our definition of $E(\C)\hat\otimes A$ coincides with the one given in \cite{GM21}. The $A$\=/module $\Sf_\phi(A)$ is the subset of $E(\C)\hat\otimes A$ comprised of the elements on which the left and right $A$\=/actions coincide, hence $\Sf_\phi(A)$ is the module of special functions as defined in \cite{GM21}.
\end{oss}

\begin{lemma}
    The maximal compact $A$\=/submodule of $E(\C)$ is $\exp_\phi(\K\Lambda_\phi)$, which is naturally isomorphic to $(\K\otimes_A\Lambda_\phi)/\Lambda_\phi$.
\end{lemma}

\begin{proof}
    The map $\exp_\phi$ is open because its Jacobian at all points is the identity; call $C$ its image. Since $C$ is an open $\F_q$\=/vector space, the quotient ${E(\C)}/{C}$ is a discrete $A$\=/module; in particular, the maximal compact $A$\=/submodule of $E(\C)$ is contained in $C$, which is isomorphic to ${\Lie(E)}/{\Lambda_\phi}$ as a topological $A$\=/module via $\exp_\phi$.

    Endow $\Lie(E)$ with the structure of topological $\K$\=/vector space described in Lemma \ref{K-action on Lie(E)}.  Since the latter is a compact $A$\=/module, and is isomorphic to ${(\K\otimes_A\Lambda_\phi)/\Lambda_\phi}$, this concludes the proof.
\end{proof}

Endow $\Lie(E)$ with the structure of topological $\K$\=/vector space described in Lemma \ref{K-action on Lie(E)}; the following lemma is fundamental to describe the functor $\Sf_\phi$.

\begin{lemma}\label{lemma new}
The quotient of the $\K$\=/vector subspace $\K\Lambda_\phi\subseteq\Lie(E)$ by $\Lambda_\phi$ is the maximal compact $A$\=/submodule of ${\Lie(E)}/{\Lambda_\phi}$.
\end{lemma}

\begin{proof}
    Since $\K\Lambda_\phi\subseteq\Lie(E)$ is finitely generated, it is closed, hence we can endow the quotient
    \[V\coloneqq \frac{\Lie(E)}{\K\Lambda_\phi}=\frac{\faktor{\Lie(E)}{\Lambda_\phi}}{\faktor{\K\Lambda_\phi}{\Lambda_\phi}}\]
    with a natural structure of topological $\K$\=/vector space. In particular, the maximal compact $A$\=/submodule of $V$ is $\{0\}$, hence any compact $A$\=/submodule of $\Lie(E)/\Lambda_\phi$ is contained in the compact $A$\=/module $\K\Lambda_\phi/\Lambda_\phi$.
\end{proof}

\begin{teo}\label{main result}
    Assume that $E$ is uniformizable. The functor $\Sf_\phi$ is naturally isomorphic to $\Hom_A(\Lambda_\phi^*\otimes_A\Omega,\_)$; moreover, the universal object in $E(\C)\hat\otimes\Hom_A(\Lambda_\phi,\Omega)$ corresponds to the map \[\widehat{\Hom_A(\Lambda_\phi,\Omega)}\cong{\K\Lambda_\phi}/{\Lambda_\phi}\to E(\C)\] sending the projection of $c\in\K\Lambda_\phi$ to $\exp_\phi(c)$.
\end{teo}

\begin{proof}
    Since $E$ is uniformizable, the exponential map induces and isomorphism of topological $A$\=/modules between ${\Lie(E)}/{\Lambda_\phi}$ and $E(\C)$, which sends $\K\Lambda_\phi/\Lambda_\phi$ to $\exp_\phi(\K\Lambda_\phi)$. By Lemma \ref{lemma new}, $\exp_\phi(\K\Lambda_\phi)$ is the maximal compact $A$\=/submodule of $E(\C)$
    
    For any discrete $A$\=/module $M$, $\hat{M}$ is compact, hence for any function in \[\Sf_\phi(M)=\Hom_A^{\mathrm{cont}}(\hat{M},E(\C)),\] its image is contained in $\exp_\phi(\K\Lambda_\phi)$. In particular, we have the following natural isomorphisms:
    \begin{align*}
        \Sf_\phi(M)&=\Hom_A^\mathrm{cont}\left(\hat{M},\exp_\phi(\K\Lambda_\phi)\right)\\
        &\cong\Hom_A^\mathrm{cont}\left(\hat{M},\faktor{\K\Lambda_\phi}{\Lambda_\phi}\right)\\
        &\cong\Hom_A(\Lambda_\phi^*\otimes_A\Omega,M),
    \end{align*}
    where we used Lemma \ref{duality tensor hom 1} to prove the last isomorphism.
    
    Setting $M\coloneqq\Lambda_\phi^*\otimes_A\Omega$, and following the identity along the chain of isomorphisms, we deduce that the universal object \[\omega_\phi\in\Hom_A^\mathrm{cont}\left({\K\Lambda_\phi}/{\Lambda_\phi},E(\C)\right)\] is the continuous $A$\=/linear map sending the projection of $c\in\K\Lambda_\phi$ to $\exp_\phi(c)$.
\end{proof}

For the sake of completeness, let's prove a statement which does not assume uniformizability.

\begin{teo}\label{non uniformizable case}
    If we restrict the functor $\Sf_\phi$ to the subcategory of torsionfree $A$\=/modules, it is naturally isomorphic to $\Hom_A(\Hom_A(\Lambda_\phi,\Omega),\_)$; moreover, the universal object in $E(\C)\hat\otimes\Hom_A(\Lambda_\phi,\Omega)$ corresponds to the map \[\widehat{\Hom_A(\Lambda_\phi,\Omega)}\cong{\K\Lambda_\phi}/{\Lambda_\phi}\to E(\C)\] sending the projection of $c\in\K\Lambda_\phi$ to $\exp_\phi(c)$.
\end{teo}
\begin{proof}
    The map $\exp_\phi$ is open because its Jacobian at all points is the identity; call $C$ its image. Since $C$ is an open $\F_q$\=/vector space, the quotient ${E(\C)}/{C}$ is a discrete $A$\=/module.
    
    A discrete $A$\=/module $M$ is torsionfree if and only if it has no nontrivial compact submodules; in this case, $\hat{M}$ is a compact $A$\=/module with no nontrivial discrete quotients. In particular, for any function $f\in \Sf_\phi(M)=\Hom_A^\mathrm{cont}(\hat{M},E(\C))$, its projection onto ${E(\C)}/{C}$ is trivial, hence the image of $f$ must be contained in $C$. The rest of the proof is the same as Theorem \ref{main result} up to substituting $E(\C)$ with $C$.
\end{proof}
\begin{Def}
    We define the \emph{universal Anderson eigenvector} \[\omega_\phi\in\C\hat\otimes\Hom_A(\Lambda_\phi,\Omega)\] as the universal object of the functor $\Sf_\phi$.
\end{Def}
As a corollary, we can describe the isomorphism class of the module of special functions $\Sf_\phi(A)$ for any Anderson module $E$, as already done by Gazda and Maurischat (\cite[Thm. 3.11]{GM21}).
\begin{cor}\label{cor Gazda}
    The following isomorphism of $A$\=/modules holds:
    \[\Sf_\phi(A)=\{\omega\in E(\C)\hat\otimes A\mid\phi_a(\omega)=(1\otimes a)\omega\text{ for all } a\in A\}\cong\Omega^*\otimes_A\Lambda_\phi.\]
\end{cor}

\begin{oss}\label{oss universal special function}
    Fix an $\F_q$\=/basis $(\mu_i)_i$ of the discrete $A$\=/module $\Hom_A(\Lambda_\phi,\Omega)$, with $(\mu_i^*)_i$ dual basis of ${\K\Lambda_\phi}/{\Lambda_\phi}$. By Proposition \ref{duality tensor hom 2} we can express the universal object in the following alternative way as an element of $E(\C)\hat\otimes\Hom_A(\Lambda_\phi,\Omega)$:
    \[\omega_\phi=\sum_i\exp_\phi(\mu_i^*)\otimes\mu_i,\]
    where by slight abuse of notation we considered $\exp_\phi$ as a map from ${\K\Lambda_\phi}/{\Lambda_\phi}$ to $E(\C)$.
\end{oss}

\section{Proof of a conjecture of Gazda and Maurischat}\label{section question}

\begin{Def}
    An Anderson module $(E,\phi)$ of dimension $1$ is called \emph{Drinfeld module}.
\end{Def}

It's known that all Drinfeld modules are uniformizable (see \cite{Goss98}[Prop. 2.13]).
We apply the results of Section \ref{section special functions} in the context of Drinfeld module $(E,\phi)$ of rank $1$, with the further assumption that $\infty\in X(\F_q)$, to answer a question posed by Gazda and Maurischat in \cite{GM21}. For simplicity, we will assume $E=\G_a$, so that $E(\C)=\C$.

\begin{oss}
    Under this assumption, if we denote by $\tau:\C\to\C$ the Frobenius endomorphism, the algebra of endomorphisms of $E$ as an $\F_q$\=/module scheme over $\C$ is canonically isomorphic to the noncommutative ring $\C[\tau]$, where $c^q \tau=\tau c$.

    Similarly, we can think of the action of $A$ on $E$ as a ring homomorphism ${\phi:A\to\C[\tau]}$ sending $a$ to $\phi_a=\sum_i a_i\tau^i$, and of the exponential $\exp_\phi:\C\to\C$ as an element of the noncommutative ring of power series $\C[\![\tau]\!]$.
\end{oss}

In this case, the properties of power series over $\C$ allow us to express the exponential function as an infinite product as follows (see for example \cite[Section 4.2]{Goss98}).

\begin{prop}
    Let $(\G_a,\phi)$ be a Drinfeld module. The exponential function, as a power series in $\C[\![z]\!]$, has the following product expansion:
    \[\exp_\phi(z)=z\prod_{\lambda\in\Lambda_\phi\setminus\{0\}}\left(1-\frac{z}{\lambda}\right).\]
\end{prop}

Let $f_\phi\in\Frac( A_\C)$ denote the shtuka function associated to the Drinfeld module (see \cite{Thakur93}, \cite[Def. 7.11.2]{Goss98}). The following property holds (see \cite[Lemma 3.6]{ANDTR17a} and \cite[Rmk. 3.10]{ANDTR17a}, or \cite[Prop. 3.18]{GM21}).

\begin{prop}
    For all $\omega\in\C\hat\otimes A$: \[\omega\in \Sf_\phi(A)\iff(\tau\otimes1)\omega=f\omega.\]
\end{prop}

In particular, if there is some $\omega\in \Sf_\phi(A)$ which is an invertible element of the ring $\C\hat\otimes A$, for all $\omega'\in \Sf_\phi(A)$ we have \[(\tau\otimes1)\left(\frac{\omega'}{\omega}\right)=\frac{\omega'}{\omega},\] i.e. $\frac{\omega'}{\omega}\in\F_q\otimes A$, hence $\Sf_\phi(A)=A\cdot\omega$.

The conjecture of Gazda and Maurischat in \cite{GM21} is about the converse statement.

\begin{conj}[{\cite[Question]{GM21}}]
    If $\Sf_\phi(A)\cong A$, there is some $\omega\in\Sf_\phi(A)$ which is invertible as an element of $\C\hat\otimes A$.
\end{conj} 

We answer affirmatively with Theorem \ref{teo Gazda}.

First, we prove two results to show that Pontryagin duality is well-behaved with respect to norms. For starters, we endow the space $\widehat{\K}\cong\Omega\otimes_A\K$ with a norm $|\cdot|$ such that it is a normed vector space over $(\K,\|\cdot\|)$, and for any ideal $J\subseteq A$ we use the same notation for the induced norm on the quotient $\hat{J}$; note that, since $\widehat{\K}$ has dimension $1$ as a $\K$\=/vector space, $|\cdot|$ is unique up to a scalar factor in $\R_{>0}$.

\begin{prop}\label{norm pairing}
    There is some scalar factor $\alpha\in \R_{>0}$ such that, for all $f\in\widehat{\K}\setminus\{0\}$, 
    $\alpha|f|^{-1}$ is the minimum norm of the elements of the closed subspace $\K\setminus f^{-1}(0)\subseteq\K$.
\end{prop}
\begin{proof}
    Let $t\in\K$ be a uniformizer: since $\infty\in X(\F_q)$, we can identify $\K$ with $\F_q(\!(t)\!)$, where if the series \[p(t)=\sum_{i\in\Z}\lambda_i t^i\in\F_q(\!(t)\!)\] has leading term $\lambda_k t^k$, its norm is $q^{-k}$. Consider the function $dt\in\widehat{\F_q(\!(t)\!)}$ which sends $p(t)$ as defined above to $\lambda_{-1}$: under the identification $\K=\F_q(\!(t)\!)$, we have \[\widehat{\K}=\F_q(\!(t)\!)dt,\] and up to a scalar factor in $\R_{>0}$ we can assume $|dt|=q^{-1}$.
    
    Take $\mu\in\F_q(\!(t)\!)dt\setminus\{0\}$ with leading term $b_k t^k dt$, so that $|\mu|=q^{-k-1}$: if $p\in\F_q(\!(t)\!)$ has $\|p\|< q^{k+1}$, its leading term has degree at least $-k$, hence $\mu(p)=0$; on the other hand $\|t^{-k-1}\|=q^{k+1}$ and $\mu(t^{-k-1})=b_k\neq0$. In particular, $q^{k+1}=|\mu|^{-1}$ is the minimum norm of the elements in $\F_q(\!(t)\!)\setminus\mu^{-1}(0)$.
\end{proof}
\begin{prop}\label{monotony}
Let $J\subseteq A$ be a nonzero ideal and fix an $\F_q$\=/basis $(a_i)_{i\in I}$ of $J$ strictly ordered by degree, with $(a_i^*)_{i\in I}$ dual basis of $\hat{J}$. The sequence $(|a_i^*|)_{i\in I}$ is strictly decreasing.
\end{prop}
\begin{proof}
We can assume $I\subseteq\Z$ to be the set of degrees of elements in $J$, and that $a_i$ has degree $i$ for all $i\in I$. For all $i\in I$ set $b_i\coloneqq a_i$, while for all $i\in\Z\setminus I$ choose some $b_i\in\K$ with valuation $-i$: since all nonzero elements of $\K$ have integer valuation, it's easy to check that every $c\in\K$ can be expressed in a unique way as $\sum_{i\in\Z}\lambda_i b_i$ where $\lambda_i\in\F_q$ for all $i\in\Z$ and $\lambda_i=0$ for $i\gg0$. Denote by $(b_i^*)_{i\in\Z}$ the sequence in $\widehat{\K}$ determined by the property $b_i^*(b_j)=\delta_{i,j}$ for all $i,j\in\Z$. By Proposition \ref{norm pairing}, up to rescaling $|\cdot|$ by some positive real factor, we have for all $i\in\Z$:
\begin{align*}
    |b_i^*|^{-1}=\min\{\|c\|\text{ s.t. }c\in\K\text{ and }b_i^*(c)\neq0\}=&\min\left\{\left\|\sum_{j\in\Z}\lambda_j b_j\right\|\text{ s.t. }\lambda_i\neq0\right\}=\|b_i\|.
\end{align*}

Let's prove that any $c\in\widehat{\K}$ can be expressed in a unique way as a series $\sum_{i\in\Z}\lambda_i b_i^*$ with $\lambda_i\in\F_q$ for all $i$ and $\lambda_i=0$ for $i\ll0$. To prove uniqueness, we have:
\[c=\sum_{i\in\Z}\lambda_i b_i^*\Leftrightarrow c(b_j)=\left(\sum_{i\in\Z}\lambda_i b_i^*\right)(b_j)\forall j\in\Z\Leftrightarrow c(b_j)=\lambda_j \forall j\in\Z.\]
To prove existence, since $c$ is continuous, $c(b_j)=0$ for $j\ll0$, and since the sequence $(|b_j^*|)_{j\in\Z}=(\|b_j\|^{-1})_{j\in\Z}$ is strictly decreasing and tends to $0$, the series $\sum_{i\in\Z}c(b_i) b_i^*$ converges in $\widehat{\K}$.

For any $c\in\widehat{\K}$, call $\overline{c}$ its restriction to $J$, in $\hat{J}$. Since $(b_i)_{i\in I}=(a_i)_{i\in I}$ is an $\F_q$\=/basis of $J$, we have $\overline{b_i^*}=a_i^*$ if $i\in I$, and $\overline{b_i^*}=0$ otherwise. For all $i\in I$, we have:
\[|a_i^*|=\min\{|c|\text{ s.t. }\overline{c}=a_i^*\}=\min\left\{\left|\sum_{j\in\Z}\lambda_j b_j^*\right|\text{ s.t. }\lambda_j=\delta_{i,j}\text{ for all }j\in I\right\}=|b_i^*|,\]
which is equal to $\|a_i\|^{-1}$.
\end{proof}

\begin{teo}\label{teo Gazda}
    Assume that $\Sf_\phi(A)$ is free of rank $1$. Then, there is a special function in $\Sf_\phi(A)$ which is invertible as an element of $\C\hat\otimes A$.
\end{teo}
\begin{proof}
    The assumption implies $\Lambda_\phi\cong\Omega$.
    Fix an $\F_q$\=/basis $(a_i)_{i\in I}$ of $A$ like in the proof of Proposition \ref{monotony}, with $a_0=1$, and let $(a_i^*)_{i\in I}$ be the dual basis of its Pontryagin dual \[\hat{A}\cong{\Omega\otimes_A\K}/{\Omega}\cong{\K\Lambda_\phi}/{\Lambda_\phi}.\]
    
    By Remark \ref{oss universal special function}, we can write the universal Anderson eigenvector as an infinite series \[\omega_\phi=\sum_i\exp_\phi(a_i^*)\otimes a_i\in\C\hat\otimes A.\] To prove it is invertible, it suffices to show that, for all $i\geq1$, \[\|\exp_\phi(a_0^*)\|>\|\exp_\phi(a_i^*)\|:\] indeed, if this is the case, and we set \[\omega\coloneqq(\exp_\phi(a_0^*)^{-1}\otimes1)\omega_\phi,\] the element $1-\omega\in\C\hat\otimes A$ has norm less than $1$, hence the series $\sum_{n\geq0}(1-\omega)^n$ converges in $\C\hat\otimes A$, and is an inverse to $1-(1-\omega)=\omega$.
    
    For all indices $i$, choose a lift $c_i\in\K\Lambda_\phi\subseteq\C$ of $a_i^*\in{\K\Lambda_\phi}/{\Lambda_\phi}$ with the least possible norm, so that $\|c_i\|=|a_i^*|$; in particular, since $\Lambda_\phi$ has rank $1$, there are no $\lambda\in\Lambda_\phi$ such that $\|\lambda\|=\|c_i\|$, so we have:
    \[\|\exp_\phi(a_i^*)\|=\|c_i\|\prod_{\lambda\in\Lambda_\phi\setminus\{0\}}\left\|1-\frac{c_i}{\lambda}\right\|=\|c_i\|\prod_{\substack{\lambda\in\Lambda_\phi\setminus\{0\}\\\|\lambda\|\leq\|c_i\|}}\left\|1-\frac{c_i}{\lambda}\right\|=\|c_i\|\prod_{\substack{\lambda\in\Lambda_\phi\setminus\{0\}\\\|\lambda\|<\|c_i\|}}\left\|\frac{c_i}{\lambda}\right\|.\]
    Since by Proposition \ref{monotony} the sequence $(\|c_i\|)_i$ is strictly decreasing, from the previous equality we deduce that the sequence $(\|\exp_\phi(a_i^*)\|)_i$ is also strictly decreasing. In particular, $\|\exp_\phi(a_0^*)\|>\|\exp_\phi(a_i^*)\|$ for all $i\geq1$.
\end{proof}

\section{Some remarkable identities\texorpdfstring{ in $\C$}{}}\label{section technical lemmas}

In this section, we focus on a Drinfeld module $\phi$ of arbitrary rank. We want to adapt the approach of the previous section to adjoint Drinfeld modules; to do so, we rely on a useful result by Poonen (\cite[Thm. 10]{Poonen96}).

\subsection{Poonen duality}\label{subsection Poonen}

Let $\C[\![\tau,\tau^{-1}]\!]$ denote the $\C[\tau,\tau^{-1}]$\=/bimodule of unbounded formal power series in $\tau$.
For any formal series $s=\sum_i s_i\tau^i\in\C[\![\tau,\tau^{-1}]\!]$, we define its adjoint as \[s^*\coloneqq\sum_i\tau^{-i}s_i=\sum_i s_i^{q^{-i}}\tau^{-i}\in\C[\![\tau,\tau^{-1}]\!].\]

\begin{oss}
    Despite the notation for the space $\C[\![\tau,\tau^{-1}]\!]$, its structure of $\C[\tau,\tau^{-1}]$\=/bimodule cannot be extended to a ring structure.
\end{oss}

\begin{oss}
    It's easy to show that, since $\exp_\phi$ has an infinite radius of convergence, the adjoint exponential $\exp_\phi^*\in\C[\![\tau^{-1}]\!]$ also converges everywhere on $\C$.
\end{oss}

We follow a construction due to Poonen, who proved a duality result of central importance to this section (\cite[Thm. 10]{Poonen96}).
\begin{lemma}\label{lemma g_beta}
    For all $\beta\in\ker(\exp_\phi^*)\setminus\{0\}$, there is a unique element $g_\beta\in\C[\![\tau]\!]$ such that $(1-\tau)g_\beta=\beta\exp_\phi$. Moreover, $g_\beta$ has infinite radius of convergence.
\end{lemma}
\begin{proof}
    Let's set $h:=\sum_{i\geq0}\tau^i$; since $h(1-\tau)=1$, the defining property of $g_\beta$ is equivalent to the identity $g_\beta=h\beta\exp_\phi$: this implies both existence and uniqueness. If we call $e_i$ the $i$\=/th coefficient of $\exp_\phi$ and $c_i$ the $i$\=/th coefficient of $g_\beta$, from the identity $g_\beta=h\beta\exp_\phi$ we get the following:
    \[c_k=\sum_{i=0}^k\beta^{q^i}e_{k-i}^{q^i}\text{ for all }k\in\Z_{\geq0}\Longrightarrow\lim_k c_k^\frac{1}{q^k}=\lim_k\sum_{i=0}^k e_i^{q^{-i}}\beta^{q^{-i}}=\exp_\phi^*(\beta)=0,\]
    hence the radius of convergence of $g_\beta$ is infinite.
\end{proof}
\begin{oss}
    Since $\ker((1-\tau)g_\beta)=\ker(\beta\exp_\phi)=\Lambda_\phi$, $g_\beta|_{\Lambda_\phi}$ has image in $\F_q$.
\end{oss}

By convention, we set $g_0=0$.

\begin{teo}[{\cite[Thm. 10]{Poonen96}}]\label{Poonen96}
The function $\ker(\exp_\phi^*)\to\hat\Lambda_\phi$ sending $\beta$ to $g_\beta|_{\Lambda_\phi}$ is an $A$\=/linear homeomorphism, where $A$ acts via $\phi^*$ on the left hand side.
\end{teo}

The following proposition, which is proven in Section \ref{section technical lemmas}, can be viewed as an explicit algebraic formula for the inverse of the isomorphism in Theorem \ref{Poonen96}.

\begin{prop*}[Prop. \ref{identity 1}]
    For all $\beta\in\ker(\exp_\phi^*)\setminus\{0\}$, the following identity holds in $\C$:
    \[\beta=-\sum_{\lambda\in\Lambda_\phi\setminus\{0\}}\frac{g_\beta(\lambda)}{\lambda}.\]
\end{prop*}

In the following subsection, we include some technical lemmas necessary for the proof of Proposition \ref{identity 1}. With the same lemmas we are also able to prove the following proposition, where \[\exp_\phi=\sum_{i\geq0}e_i\tau^i,\] and we denote by \[\log_\phi=\sum_{i\geq0}l_i\tau^i\] its inverse in $\C[\![\tau]\!]$.

\begin{prop*}[Prop. \ref{identity 2}]
    For all integers $k$, for all $c\in\K\setminus\{0\}$ with $\|c\|\leq q^{\frac{k-1}{r}}$, the following identity holds in $\C$:
    \[\sum_{\lambda\in\Lambda_\phi\setminus\{0\}}\frac{\exp_\phi(c\lambda)}{\lambda^{q^k}}=-\sum_{j=0}^k e_jl_{k-j}^{q^j}c^{q^j},\]
    where by convention the summation on the right hand side is $0$ if $k<0$.
\end{prop*}

\subsection{Lattices}

Throughout this subsection, $C$ will be a complete normed $\K$\=/vector space (with non-archimedean norm).

\begin{Def}\label{def Lambda_m}
    An infinite $\F_q$\=/linear subspace $V\subseteq C$ is a \textit{lattice} if for any positive real number $r$ there are finitely many elements of $V$ of norm at most $r$.
    
    An \textit{ordered basis} of $V$ is a sequence $(v_i)_{i\geq1}$ with the following property: for all $m\geq1$, $v_m$ is an element of $V\setminus\Span_{\F_q}(\{v_i\}_{i<m})$ of least norm.

    We call the sequence of real numbers $(\|v_i\|)_{i\geq1}$ a \textit{norm sequence} of $V$.
\end{Def}

The next two results aim to justify the nomenclature "ordered basis" and to prove that the norm sequence does not depend on the choice of the ordered basis.

\begin{oss}
    If $V\subseteq C$ is a lattice, every subset $S\subseteq V$ has an element of least norm. In particular, we can construct an ordered basis of $V$ by recursion.
\end{oss}

\begin{lemma}
    If $(v_i)_{i\geq1}$ is an ordered basis of a lattice $V\subseteq C$, it is a basis of $V$ as an $\F_q$\=/vector space.
\end{lemma}
\begin{proof}
    For all $m\geq1$ \[v_m\not\in \Span_{\F_q}(\{v_i\}_{i<m}),\] hence the $v_i$'s are $\F_q$\=/linearly independent. Since for all $r\in\R$ there is a finite number of elements of $V$ with norm at most $r$, the norm sequence $(\|v_i\|)_{i\geq1}$ tends to infinity; in particular, for all $v\in V$ there is an integer $m$ such that $\|v_m\|>\|v\|$, so \[v\in \Span_{\F_q}(\{v_i\}_{i<m})\] by construction of $v_m$.
\end{proof}
\begin{prop}\label{norm sequence}
    If $(v_i)_{i\geq1}$ is an ordered basis of a lattice $V\subseteq C$, and $(v'_i)_{i\geq1}$ is a sequence of elements in $V$ that are $\F_q$\=/linearly independent and with weakly increasing norm, then \[\|v'_i\|\geq\|v_i\|\text{ for all }i\geq1.\] In particular, the norm sequence of $V$ does not depend on the chosen ordered basis of $V$.
\end{prop}
\begin{proof}
    By contradiction, assume $\|v'_m\|<\|v_m\|$ for some $m$. Then for all $i\leq m$ we have \[\|v'_i\|\leq\|v'_m\|<\|v_m\|,\] so $v'_i\in \Span_{\F_q}(\{v_j\}_{j<m})$ by construction of $v_m$; since $\{v'_i\}_{i\leq m}$ is a set of $m$ $\F_q$\=/linearly independent vectors and $\dim_{\F_q}\Span_{\F_q}(\{v_j\}_{j<m})=m-1$, we have reached a contradiction. If we take $(v'_i)_i$ to be another ordered basis, by this reasoning we get both \[\|v'_m\|\geq\|v_m\|\text{ and }\|v_m\|\geq\|v'_m\|,\] hence the norm sequence is independent from the choice of the ordered basis.
\end{proof}

Finally, we show that the norm sequence is reasonably well behaved with regard to subspaces.

\begin{lemma}\label{subspace ordered basis}
Let $W\subseteq V\subseteq C$ be lattices. The norm sequence $(s_i)_{i\geq1}$ of $W$ is a subsequence of the norm sequence $(r_i)_{i\geq1}$ $V$. 

Moreover, if $\dim_{\F_q}{V}/{W}=n<\infty$, for $i\gg0$ we have $r_i=s_{i+n}$.
\end{lemma}
\begin{proof}
Let $(w_i)_{i\geq1}$ be an ordered basis of $W$. Let's construct an ordered basis $(v_i)_{i\geq1}$ of $V$ recursively in the following way. For all $k\geq1$ let $f(k)$ be the least integer such that \[w_{f(k)}\not\in\Span_{\F_q}(\{v_i\}_{i<k}),\] and let \[v_k'\in V\setminus\Span_{\F_q}(\{v_i\}_{i<k})\] be an element of least norm. If $\|v_k'\|<\|w_{f(k)}\|$, we set $v_k\coloneqq v_k'$, otherwise we set $v_k\coloneqq w_{f(k)}$. 

By construction $(v_k)_{k\geq1}$ is an ordered basis of $V$, so we only need to show that for all $j\geq1$ there is some $k\geq1$ such that $v_k=w_j$. By contradiction, let $j$ be the first integer such that this does not happen, and let $k$ be the greatest integer such that \[w_j\not\in\Span_{\F_q}(\{v_i\}_{i<k}),\] which exists because $(v_i)_{i\geq1}$ is a basis of $V$. This means that $w_j=\alpha v_k+v$ for some $v\in\Span_{\F_q}(\{v_i\}_{i<k})$ and some constant $\alpha\in\F_q^\times$, and since $v_k\neq w_j$, by our algorithm we must have $\|v_k\|<\|w_j\|$; as a consequence \[\|v\|=\|w_j-\alpha v_k\|=\|w_j\|>\|v_k\|,\] which is a contradiction because, since $(v_i)_{i\geq1}$ is an ordered basis, $\|v_k\|\geq\|v_i\|$ for all $i<k$, hence $\|v_k\|\geq\|v\|$.

If $\dim_{\F_q}{V}/{W}=n<\infty$, since the basis $\{v_i\}_{i\geq1}$ of $V$ extends the basis $\{w_i\}_{i\geq1}$ of $W$, there are exactly $n$ elements of the former which are not contained in the latter. Since, taking the order into account, $(w_i)_{i\geq1}$ is a subsequence of $(v_i)_{i\geq1}$, for $i\gg0$ we have $v_i=w_{i+n}$, hence $r_i=s_{i+n}$.
\end{proof}

Recall that $e$ is the degree of $\infty\in X$.

\begin{lemma}\label{norm bound}
    Let $V\subseteq\C$ be a lattice which is also a projective $A$\=/submodule of finite rank $r$, and let $(s_i)_{i\geq1}$ be its norm sequence. Then:
    \begin{itemize}
        \item for all $i\gg0$, \[s_{i+er}=q^e\cdot s_i;\]
        \item for all $k\in\Z$, for all $i\gg0$, \[\frac{s_{i+k}}{s_i}\leq q^{e\left\lceil\frac{k}{er}\right\rceil};\]
        \item for all $k\in\Z$, for infinitely many $i$, \[\frac{s_{i+k}}{s_i}\leq q^\frac{k}{r}.\]
    \end{itemize}
\end{lemma}
\begin{proof}
    We can choose $a,b\in A\setminus\{0\}$ such that $\deg(b)=\deg(a)+e$. Fix an ordered basis $(v_i)_{i\geq1}$ of $V$: obviously, $(av_i)_{i\geq1}$ and $(bv_i)_{i\geq1}$ are ordered bases respectively of $aV$ and $bV$. Since $V$ has rank $r$, \[\dim_{\F_q}{V}/{aV}=r\deg(a)\text{ and }\dim_{\F_q}{V}/{bV}=r\deg(b),\] therefore by Lemma \ref{subspace ordered basis} we have \[\|v_i\|=\|av_{i-r\deg(a)}\|=\|bv_{i-r\deg(b)}\|\text{ for }i\gg0.\]
    Rearranging the terms, we get that, for $i\gg0$:
    \[\|v_{i-r\deg(a)}\|=\|a\|^{-1}\|b\|\|v_{i-r\deg(b)}\|=q^e\|v_{i-r\deg(b)}\|.\]
    Shifting the indices we get \[\|v_i\|=q^e\|v_{i-r(\deg(b)-\deg(a))}\|=q^e\|v_{i-er}\|\text{ for }i\gg0,\] which is the first statement.

    For all $k\in\Z$, since the norm sequence is weakly increasing, we have the following inequality for $i\gg0$:
    \[\frac{s_{i+k}}{s_i}\leq\frac{s_{i+er\left\lceil\frac{k}{er}\right\rceil}}{s_i}= q^{e\left\lceil\frac{k}{er}\right\rceil}.\]
    Moreover, for all $i\gg0$:
\[\prod_{j=0}^{er-1}\frac{s_{i+k(j+1)}}{s_{i+kj}}=\frac{s_{i+ker}}{s_i}=\begin{dcases}
    \prod_{j=0}^{k-1}\frac{s_{i+e(j+1)r}}{s_{i+ejr}}=q^{ek}&\text{if }k\geq0\\
    \prod_{j=k}^{-1}\frac{s_{i+e(j+1)r}}{s_{i+ejr}}=q^{ek}&\text{if }k<0,
\end{dcases}\]
hence at least one of the factors on the left hand side has norm at most $q^{\frac{k}{r}}$; this implies that the inequality \[\frac{s_{i+k}}{s_i}\leq q^\frac{k}{r}\] holds for infinitely many values of $i$.

\end{proof}

\subsection{Estimation of the coefficients of \texorpdfstring{$g_\beta$ and $\exp_\phi$}{gβ and exp}}

The following is a combinatorial result similar to the well known "vanishing lemma" (see \cite[Lemma 8.8.1]{Goss98}).
\begin{lemma}[{\cite[Lemma 5.6]{Ferraro}}]\label{S_n,k}
    Call $S_{n,d}(x_1,\dots,x_n)\in\F_q[x_1,\dots,x_n]$ the sum of the $d$\=/th powers of all the homogeneous linear polynomials. Assume that the coefficient of monomial $x_1^{d_1}\cdots x_n^{d_n}$ in the expansion of $S_{n,d}(x_1,\dots,x_n)$ is nonzero: then, for all $1\leq j\leq n$, \[\sum_{i=1}^j d_i\geq q^j-1.\] In particular, if $d<q^n-1$, $S_{n,d}=0$.
\end{lemma}
    

\begin{Def}
    For a lattice $V\subseteq \C$, for all integers $i\geq0$ we define:
    \[e_{V,i}\coloneqq\sum_{\substack{I\subseteq V\setminus\{0\}\\|I|=q^i-1}}\prod_{v\in I}v^{-1}\]
    (by convention, $e_{V,0}=1$).
\end{Def}
\begin{oss}\label{remark e_V}
    For all $c\in\C$, since $V\subseteq\C$ is a lattice, the infinite product $c\prod_{v\in V}\left(1-\frac{c}{v}\right)$ converges, and is equal to $\sum_{n\geq0} e_{V,n}c^{q^n}$. In particular, \[\sum_{n\geq0}e_{V,n}x^{q^n}\in\C[\![x]\!]\] is the only power series with infinite radius of convergence and with leading coefficient $1$ such that its zeroes are simple and coincide with $V$.
\end{oss}

\begin{lemma}\label{coefficient bounds}
    Fix a lattice $V\subseteq\C$, with norm sequence $(r_i)_{i\geq1}$. Fix an ordered basis $(v_i)_{i\geq1}$ and call $V_m\coloneqq\Span_{\F_q}(\{v_i\}_{i\leq m})$ for all $m\geq0$. We have:
    \begin{itemize}
        \item for all $k\geq0$:
        \[\|e_{V,k}\|\leq\prod_{i=1}^k r_i^{q^{i-1}-q^i};\]
        \item for all $m>0$, for all $k>0$:
        \[\left\|\sum_{v\in V_m}v^{q^k-1}\right\|\begin{dcases}=0&\text{if }k<m\\
        \leq r_m^{q^k-q^m}\prod_{i=1}^m r_i^{q^i-q^{i-1}}&\text{if }k\geq m.\end{dcases}\]
    \end{itemize}
\end{lemma}
\begin{proof}
    For the first part, if $k=0$ then $e_{V,k}=1$, so there is nothing to prove. If $k>0$, we have:
    \[\|e_{V,k}\|=\left\|\sum_{\substack{I\subseteq V\setminus\{0\}\\|I|=q^k-1}}\prod_{v\in I}v^{-1}\right\|\leq\max_{\substack{I\subseteq V\setminus\{0\}\\|I|=q^k-1}}\left\|\prod_{v\in I}v^{-1}\right\|=\left\|\prod_{v\in V_k}v^{-1}\right\|=\prod_{i=1}^k r_i^{q^{i-1}-q^i}.\]
    For the second part, note that the element whose norm we are trying to estimate is equal to $S_{m,q^k-1}(v_1,\dots,v_m)$, in the notation of Lemma \ref{S_n,k}. By that lemma, if $k<m$, the element is zero, otherwise we have the following inequality:
    \[\|S_{m,q^k-1}(v_1,\dots,v_m)\|\leq\max_{\substack{d_1,\dots,d_m\\d_1+\cdots+d_m=q^k-1\\ \forall j\;d_1+\cdots+d_j\geq q^j-1}}\left\|v_1^{d_1}\cdots v_m^{d_m}\right\|.\]
    It's easy to see that the maximum norm of the product $v_1^{d_1}\cdots v_m^{d_m}$ under the specified conditions is obtained when we set $d_i=q^i-q^{i-1}$ for $i<m$ and $d_m=q^k-q^{m-1}$, therefore we get the desired inequality.
\end{proof}

\begin{oss}
    Since $ \Lambda_\phi\subseteq\K \Lambda_\phi$ is discrete and $\K \Lambda_\phi\cong K_\infty^r$ is locally compact, $ \Lambda_\phi$ is a lattice of $\C$. Moreover, $e_{ \Lambda_\phi,n}$ is exactly the coefficient of $\tau^n$ of the exponential function $\exp_\phi\in\C[\![\tau]\!]$.
\end{oss}

Recall the definition of $g_\beta$ given at the start of the section for all $\beta\in\ker(\exp_\phi^*)$.

\begin{lemma}\label{V_beta}
    For all $\beta\in\ker(\exp_\phi^*)\setminus\{0\}$, $\ker(g_\beta)$ is an $\F_q$\=/vector subspace of $ \Lambda_\phi$ of codimension $1$. In particular, $g_\beta=\beta\sum_{n\geq0}e_{\ker(g_\beta),n}\tau^n$.
\end{lemma}
\begin{proof}
    If $c\in \ker(g_\beta)$ then \[\exp_\phi(c)=\beta^{-1}(1-\tau)(g_\beta(c))=0,\] hence $c\in \Lambda_\phi$. Moreover, $g_\beta|_{\Lambda_\phi}$ is an $\F_q$\=/linear function with image in $\F_q$, hence its kernel $V_\beta$ has codimension at most $1$ in $\Lambda_\phi$. It is exactly $1$ because $g_\beta|_{\Lambda_\phi}$ is not identically zero by Theorem \ref{Poonen96}.

    From the identity \[(1-\tau)\circ g_\beta=\beta\exp_\phi,\] since the zeroes of $\exp_\phi$ are simple, we deduce the same for the zeroes of $g_\beta$, therefore \[g_\beta=c_\beta \sum_{n\geq0}e_{\ker(g_\beta),n}\tau^n\] for some constant $c_\beta\in\C$ by Remark \ref{remark e_V}. Finally, from the same identity we deduce that the coefficient of $\tau$ in the expansion of $g_\beta$ is $\beta$, hence $c_\beta=\beta$.
\end{proof}

\subsection{Proof of the identities}
We can now prove the main propositions of this section.

\begin{prop}\label{identity 1}
    For all $\beta\in\ker(\exp_\phi^*)$, the following identity holds in $\C$:
    \[\beta=-\sum_{\lambda\in\Lambda_\phi\setminus\{0\}}\frac{g_\beta(\lambda)}{\lambda}.\]
\end{prop}

\begin{proof}

The series converges for all $\beta\in\ker(\exp_\phi^*)$ because the denominators belong to the lattice $\Lambda_\phi$ and the numerators to $\F_q$. For $\beta=0$ the identity is obvious, hence we can assume $\beta\neq0$.
Fix an ordered basis $(\lambda_i)_{i\geq1}$ of $\Lambda_\phi$ and define \[\Lambda_m\coloneqq\Span_{\F_q}(\{\lambda_i\}_{i\leq m})\text{ for all }m\geq0.\] By Lemma \ref{V_beta}, $\ker(g_\beta)\subseteq\Lambda_\phi$ has codimension $1$, hence by Lemma \ref{subspace ordered basis}, if we denote by $(r_i)_{i\geq1}$ and $(s_i)_{i\geq1}$ the norm sequences respectively of $\Lambda_\phi$ and $\ker(g_\beta)$, there is a positive integer $N$ such that
\[s_i=\begin{dcases}r_i&\text{if }i<N\\r_{i+1}&\text{if }i\geq N.\end{dcases}\]
For all $m\geq N$, we define:
    \[S_m\coloneqq\beta+\sum_{\lambda\in\Lambda_m\setminus\{0\}}\frac{g_\beta(\lambda)}{\lambda}=\beta\sum_{k\geq1}e_{\ker(g_\beta),k}\sum_{\lambda\in\Lambda_m}\lambda^{q^k-1}.\]
    By Lemma \ref{coefficient bounds}, we have:
    \begin{align*}
        \|\beta^{-1}S_m\|=&\left\|\sum_{k\geq1}e_{\ker(g_\beta),k}\sum_{\lambda\in\Lambda_m}\lambda^{q^k-1}\right\|\leq\max_{k\geq m}\left\{\|e_{\ker(g_\beta),k}\|\left\|\sum_{\lambda\in\Lambda_m}\lambda^{q^k-1}\right\|\right\}\\
        \leq&\max_{k\geq m}\left\{\left(\prod_{i=1}^k s_i^{q^{i-1}-q^i}\right)\left(r_m^{q^k-q^m}\prod_{i=1}^m r_i^{q^i-q^{i-1}}\right)\right\}\\
        =&\max_{k\geq m}\left\{\left(\prod_{i=N}^k r_{i+1}^{q^{i-1}-q^i}\right)\left(r_m^{q^k-q^m}\prod_{i=N}^m r_i^{q^i-q^{i-1}}\right)\right\}\\
        =&\max_{k\geq m}\left\{\left(\prod_{i=N}^{m} \left(\frac{r_i}{r_{i+1}}\right)^{q^{i}-q^{i-1}}\right)\left(\prod_{i=m+1}^k\left(\frac{r_m}{r_i}\right)^{q^i-q^{i-1}}\right)\right\}\\
        =&\prod_{i=N}^{m} \left(\frac{r_i}{r_{i+1}}\right)^{q^{i}-q^{i-1}}\\
        =&\left(\frac{r_N}{r_{m+1}}\right)^{q^N-q^{N-1}}\prod_{i=N+1}^{m}\left(\frac{r_i}{r_{m+1}}\right)^{q^i-2q^{i-1}+q^{i-2}}\\
        \leq&\left(\frac{r_N}{r_{m+1}}\right)^{q^N-q^{N-1}}.
    \end{align*}
Since this number tends to zero as $m$ tends to infinity, we have the following identity in $\C$:
    \[0=\lim_m S_m=\lim_m\left(\beta+\sum_{\lambda\in\Lambda_m\setminus\{0\}}\frac{g_\beta(\lambda)}{\lambda}\right)=\beta+\sum_{\lambda\in\Lambda_\phi\setminus\{0\}}\frac{g_\beta(\lambda)}{\lambda}.\tag*{\qedhere}\]
\end{proof}

Before proving the remaining Proposition, let's recall the following well-known result about the coefficients of the logarithm $\log_\Lambda=\sum_i l_i\tau^i$ (see for example the proof of \cite[Lemma 7.1.8]{Ferraro}, which holds for arbitrary rank).

\begin{lemma}\label{lemma log coeff}
    For all $i\geq1$, $l_i=-\sum_{\lambda\in\Lambda\setminus\{0\}}\lambda^{1-q^i}$.
\end{lemma}

\begin{prop}\label{identity 2}
    For all integers $k$, for all $c\in\K\setminus\{0\}$ with $\|c\|\leq q^{\frac{k-1}{r}}$, the following identity holds in $\C$:
    \[\sum_{\lambda\in\Lambda_\phi\setminus\{0\}}\frac{\exp_\phi(c\lambda)}{\lambda^{q^k}}=-\sum_{j=0}^k e_jl_{k-j}^{q^j}c^{q^j},\]
    where by convention the summation on the right hand side is $0$ if $k<0$.
\end{prop}

\begin{proof}
First of all, let's show that series on the left hand side converges. Since $\exp_\phi(\K\Lambda_\phi)$ is homeomorphic to the compact space $\faktor{\K\Lambda_\phi}{\Lambda_\phi}$, the numerators $\exp_\phi(c\lambda)$ are bounded in norm by some positive real constant. In particular, since $\Lambda_\phi\subseteq\C$ is a lattice, for any positive real number $\varepsilon$ there are finitely many $\lambda\in\Lambda_\phi$ such that $\left\|\frac{\exp_\phi(c\lambda)}{\lambda^{q^k}}\right\|>\varepsilon$, so the series converges.

Fix an ordered basis $(\lambda_i)_{i\geq1}$ of $\Lambda_\phi$, set $r_i\coloneqq\|\lambda_i\|$, and define \[\Lambda_m\coloneqq\Span_{\F_q}(\{\lambda_i\}_{i\leq m})\text{ for all }m\geq1;\] define:
    \[S_m\coloneqq\sum_{\lambda\in\Lambda_m\setminus\{0\}}\frac{\exp_\phi(c\lambda)}{\lambda^{q^k}}-\sum_{\substack{0\leq j\leq k\\\lambda\in\Lambda_m\setminus\{0\}}}e_jc^{q^j}\lambda^{q^j-q^k}=\sum_{j\geq1}e_{k+j}c^{q^{k+j}}\left(\sum_{\lambda\in\Lambda_m}\lambda^{q^j-1}\right)^{q^k},\]
    where by convention we set $e_j=0$ for all $j<0$.
By Lemma \ref{coefficient bounds}, for all $m\gg0$ we have:
    \begin{align*}
        \|S_m\|=&\left\|\sum_{j\geq1}e_{k+j}c^{q^{k+j}}\left(\sum_{\lambda\in\Lambda_m}\lambda^{q^j-1}\right)^{q^k}\right\|\leq\max_{j\geq m}\left\{\|e_{k+j}\|\|c\|^{q^{k+j}}\left\|\sum_{\lambda\in\Lambda_m}\lambda^{q^j-1}\right\|^{q^k}\right\}\\
        \leq&\max_{j\geq m}\left\{\|c\|^{q^{k+j}}\left(\prod_{i=1}^{k+j} r_i^{q^{i-1}-q^i}\right)\left(r_m^{q^j-q^m}\prod_{i=1}^m r_i^{q^i-q^{i-1}}\right)^{q^k}\right\}\\
        \leq&\max_{j\geq m}\left\{\|c\|^{q^k}\left(\prod_{i=1-k}^{j} r_{i+k}^{q^{i+k-1}-q^{i+k}}\right)\left(\prod_{i=1}^j (\|c\|r_i)^{q^{i+k}-q^{i+k-1}}\right)\right\}\\
        =& C_k\cdot\max_{j\geq m}\left\{\prod_{i=m+1}^{j}\left(\frac{\|c\|r_i}{r_{i+k}}\right)^{q^{i+k}-q^{i+k-1}}\right\}\\
        &\Longrightarrow\limsup_m\|S_m\|\leq C_k\cdot\limsup_j \prod_{i=m+1}^j \left(\frac{\|c\|r_i}{r_{i+k}}\right)^{q^{i+k}-q^{i+k-1}},
    \end{align*}
where $C_k$ is a nonzero constant which depends on $k$.
Since the norms of nonzero elements of $\K$ are integer powers of $q^e$, we actually have the inequality \[\|c\|\leq q^{e\left\lfloor\frac{k-1}{er}\right\rfloor}.\] By Lemma \ref{norm bound}, we have:
\begin{align*}
    \|c\|\frac{r_i}{r_{i+k}}&\leq q^{e\left\lfloor\frac{k-1}{er}\right\rfloor}\cdot q^{e\left\lceil-\frac{k}{er}\right\rceil}=q^{e\left\lfloor\frac{k-1}{er}\right\rfloor}\cdot q^{-e\left\lfloor\frac{k}{er}\right\rfloor}\leq 1\text{ for all $i$ large enough;}\\
    \|c\|\frac{r_i}{r_{i+k}}&\leq q^{e\left\lfloor\frac{k-1}{er}\right\rfloor}\cdot q^{-\frac{k}{r}}\leq q^{\frac{k-1}{r}}\cdot q^{-\frac{k}{r}}=q^{-\frac{1}{r}}<1\text{ for infinitely many values of $i$}.
\end{align*}
The first inequality implies that the limit superior on the right hand side is finite, the second inequality that it is zero.
We deduce that the sequence $\|S_m\|$ converges to $0$. If $k<0$, we get the following identity in $\C$:
\[\sum_{\lambda\in\Lambda_\phi\setminus\{0\}}\frac{\exp_\phi(c\lambda)}{\lambda^{q^k}}=\lim_m S_m=0.\]
If instead $k\geq0$, we get the following identity in $\C$:
\begin{align*}
    \sum_{\lambda\in\Lambda_\phi\setminus\{0\}}\frac{\exp_\phi(c\lambda)}{\lambda^{q^k}}&=\lim_m\left(S_m+\sum_{j=0}^k e_jc^{q^j}\sum_{\lambda\in\Lambda_m\setminus\{0\}}\lambda^{q^j-q^k}\right)\\
    &=\sum_{j=0}^{k-1} \left(e_jc^{q^j}\sum_{\lambda\in\Lambda_\phi\setminus\{0\}}\lambda^{q^j-q^k}\right)-e_kc^{q^k}\\
    &=-\sum_{j=0}^k e_jl_{k-j}^{q^j}c^{q^j}
\end{align*}
where the last equality follows from Lemma \ref{lemma log coeff}.
\end{proof}

\section{Universal dual Anderson eigenvector}\label{section zeta functions}

\subsection{Functor of dual Anderson eigenvectors}

In \cite{Ferraro}, the author proved that, if $\phi$ is a normalized Drinfeld module of rank $1$ and $\Lambda_\phi=\tilde{\pi}I$ for some $\tilde{\pi}\in\C$ and some ideal $I\subseteq A$, the following result holds (see \cite[Prop. 7.8, Prop. 7.21]{Ferraro}).
\begin{prop}\label{prop ferraro}
    Let $\zeta_I\coloneqq\sum_{a\in I\setminus\{0\}}a^{-1}\otimes a\in\C\hat\otimes A$. For all $a\in A\setminus\{0\}$:
    \[\phi^*_a\left((\tilde{\pi}^{-1}\otimes1)\zeta_I\right)=(\tilde{\pi}^{-1}\otimes a)\zeta_I.\]
\end{prop}

In other words, $(\tilde{\pi}^{-1}\otimes1)\zeta_I$ is a "dual" special function: it satisfies a functional identity analogous  to that of the special functions introduced by Anglès, Ngo Dac, and Tavares Ribeiro, with the Drinfeld module replaced by its adjoint.

Let $(\G_a,\phi)$ be a Drinfeld $A$\=/module. Let's endow the topological $\F_q$\=/vector space $\C$ with the following $A$\=/module structure: $a\in A$ sends $c\in\C$ to $\phi_a^*(c)$; to avoid confusion, we denote this topological $A$\=/module by $\mathbb{C}_\infty^{\phi^*}$.

\begin{Def}\label{Def dual Anderson eigenvectors}
Let $(\G_a,\phi)$ be a Drinfeld $A$\=/module. For any discrete $A$\=/module $M$, its set of \textit{dual Anderson eigenvectors} is defined as the $A$\=/module of continuous $A$\=/linear homomorphisms \[\Hom_A^\mathrm{cont}(\hat{M},\mathbb{C}_\infty^{\phi^*})\subseteq \mathbb{C}_\infty^{\phi^*}\hat\otimes M.\]
We denote by $\Sf_{\phi^*}:A\mbox{-}\Mod\to A\mbox{-}\Mod$ the functor that extends this map in the straightforward way.
\end{Def}

\begin{oss}\label{oss d.a.e.}
    We can write as follows the property of being a dual Anderson eigenvector $\zeta=\sum_i z_i\otimes m_i\in \mathbb{C}_\infty^{\phi^*}\hat\otimes M$. For all $a\in A$:
\begin{align*}
    \sum_i z_i\otimes am_i&=\left(\sum_j(\phi_a)_j^*\otimes 1\right)\left(\sum_i z_i\otimes m_i\right)\\&=\sum_{i,j}(\phi_a)_j^* z_i\otimes m_i\\&=\sum_{i,j}(\tau^{-j}\circ z_i\circ(\phi_a)_j)\otimes m_i.
\end{align*}
\end{oss}


\begin{teo}\label{zeta function functor}
    Let $(\G_a,\phi)$ be a Drinfeld module of rank $r$.
    The functor $\Sf_{\phi^*}$ is naturally isomorphic to $\Hom_A(\Lambda_\phi,\_)$; moreover, the universal object in $\mathbb{C}_\infty^{\phi^*}\hat\otimes\Lambda_\phi$ corresponds to the map $\hat\Lambda_\phi\cong\ker\exp^*_\phi\subseteq\mathbb{C}_\infty^{\phi^*}$ and can be expressed as \[\zeta_\phi\coloneqq-\sum_{\lambda\in\Lambda_\phi\setminus\{0\}}\lambda^{-1}\otimes\lambda.\]
\end{teo}

\begin{proof}
    Endow $\C$ with the $A$-module structure induced by the inclusion $A\subseteq\C$. The map $\exp_\phi^*:\mathbb{C}_\infty^{\phi^*}\to\C$ is a continuous $A$\=/linear morphism; for any $A$\=/module $M$, it induces a morphism $\Sf_{\phi^*}(M)\to\Hom^\mathrm{cont}_A(\hat{M},\C)$. Fix some $\zeta\in \Sf_{\phi^*}(M)$, with image $\overline{\zeta}$: since $\hat{M}$ is compact, $\overline{\zeta}(\hat{M})$ must be a compact $A$\=/submodule of $\C$, but for any $c\in\C\setminus\{0\}$ the set $A\cdot c$ is unbounded, hence $\overline{\zeta}\equiv0$. We deduce that the image of $\zeta:\hat{M}\to\mathbb{C}_\infty^{\phi^*}$ must be contained in $\ker\exp_\phi^*$, which by Theorem \ref{Poonen96} is isomorphic as a topological $A$\=/module to $\hat\Lambda_\phi$; we have the following natural isomorphisms:
    \[\Sf_{\phi^*}(M)=\Hom_A^\mathrm{cont}(\hat{M},\ker\exp_\phi^*)\cong\Hom_A^\mathrm{cont}\left(\widehat{\ker\exp_\phi^*},M\right)\cong\Hom_A(\Lambda_\phi,M),\]
    where we used Lemma \ref{duality tensor hom 1} for the second isomorphism.

    The universal object $\zeta_\phi\in\mathbb{C}_\infty^{\phi^*}\hat\otimes\Lambda_\phi$ is given by the natural morphism \[\psi:\hat\Lambda_\phi\cong\ker{\exp_\phi^*}\subseteq\mathbb{C}_\infty^{\phi^*}\] of Theorem \ref{Poonen96}, which by Proposition \ref{identity 1} sends $g\in\hat\Lambda_\phi$ to \[-\sum_{\lambda\in\Lambda_\phi\setminus\{0\}}\frac{g(\lambda)}{\lambda}.\]

    If we fix an $\F_q$\=/basis $(\lambda_i)_i$ of $\Lambda_\phi$, with $(\lambda_i^*)_i$ dual basis of $\hat\Lambda_\phi$, by Proposition \ref{duality tensor hom 2} we can write $\zeta_\phi=\sum_i\psi(\lambda_i^*)\otimes\lambda_i$, hence:
    \begin{align*}
        \zeta_\phi
    &=\sum_i\left(-\sum_{\lambda\in\Lambda_\phi\setminus\{0\}}\frac{\lambda_i^*(\lambda)}{\lambda}\right)\otimes\lambda_i\\&=-\sum_{\lambda\in\Lambda_\phi\setminus\{0\},i}\lambda^{-1}\otimes \lambda_i^*(\lambda)\lambda_i\\&=-\sum_{\lambda\in\Lambda_\phi\setminus\{0\}}\lambda^{-1}\otimes\lambda.
    \end{align*}
\end{proof}

\begin{Def}\label{def main result}
    We define the \emph{universal dual Anderson eigenvector} $\zeta_\phi\in\C\hat\otimes\Lambda_\phi$ as the universal object of the functor $\Sf_{\phi^*}$.
\end{Def}

\begin{cor}
   For all discrete $A$\=/modules $M$, $\Sf_{\phi^*}(M)$, as an $A\otimes A$\=/module, is isomorphic to $\Hom_A(\Lambda_\phi,M)$. In particular, for any $M$ we have the following equality between subsets of $\C\hat\otimes M$:
   \[\Sf_{\phi^*}(M)=\left\{\sum_{\lambda\in\Lambda_\phi\setminus\{0\}}\lambda^{-1}\otimes l(\lambda)\biggm\vert l\in\Hom_A(\Lambda_\phi,M)\right\}.\]
\end{cor}

\begin{oss}\label{oss universal dual special function}
    Fix an $\F_q$\=/basis $(\lambda_i)_i$ of the discrete $A$\=/module $\Lambda_\phi$, with $(\lambda_i^*)_i$ dual basis of $\hat\Lambda_\phi$. By Proposition \ref{duality tensor hom 2} we can express the universal object in the following alternative way as an element of $\mathbb{C}_\infty^{\phi^*}\hat\otimes\Lambda_\phi$:
    \[\zeta_\phi=\sum_i\psi(\lambda_i^*)\otimes\lambda_i,\]
    where $\psi$ denotes Poonen's isomorphism $\hat\Lambda_\phi\cong\ker(\exp_\phi^*)\subseteq\mathbb{C}_\infty^{\phi^*}$.
\end{oss}

\subsection{A convergence result for the universal Anderson eigenvector}

Let's fix a Drinfeld module $(\G_a,\phi)$ of rank $1$ and an ordered basis $(\lambda_i)_{i\geq1}$ of $\Lambda_\phi$. By Proposition \ref{duality tensor hom 2}, there is a unique sequence $(z_i)_i$ in $\C$, converging to $0$, such that we can write $\zeta_\phi=\sum_i z_i\otimes\lambda_i\in\C\hat\otimes\Lambda_\phi$. Under the assumption $\infty\in X(\F_q)$, Chung, Ngo Dac, and Pellarin proved that, for any nonnegative integer $k$, $\sum_i z_i^{q^{k}}\lambda_i$ converges to the $k$\=/th coefficient of the logarithm, while for any negative integer $k$ it converges to $0$ (\cite{CNDP21}).

We aim to generalize this result to a Drinfeld module of arbitrary rank---without any assumption on $\infty$---by exploiting the defining property of the universal Anderson eigenvector.

\begin{prop}\label{prop analyticity of zeta}
    Let $(\G_a,\phi)$ be a Drinfeld module of rank $r$, fix an $\F_q$\=/linear basis $(\lambda_i)_{i\geq1}$ of $\Lambda_\phi$, and write $\zeta_\phi=\sum_i z_i\otimes\lambda_i\in\C\hat\otimes\Lambda_\phi$. Then, for all integers $k$ the series $\sum_i z_i^{q^k}\lambda_i$ converges; moreover, if $k\geq 0$ it converges to the $k$\=/th coefficient of the logarithm $l_k$, while if $k<0$ it converges to $0$.
\end{prop}
\begin{proof}
    Let's fix $a\in A\setminus\F_q$ and fix an ordered basis $(\lambda''_i)_{i\geq1}$ of $\Lambda_\phi$. By Lemma \ref{norm bound} there is some $N$ such that, for all $i>N$, $\|\lambda''_i\|=\|a\lambda''_{i-r\deg(a)}\|$; if we define
    \[\lambda'_i\coloneqq\begin{dcases}
        \lambda''_i&\text{if }i\leq N\\
        a\lambda'_{i-r\deg(a)}&\text{if }i>N,
    \end{dcases}\]
    we have $\|\lambda''_i\|=\|\lambda'_i\|$ for all $i$, hence $(\lambda'_i)_i$ is also an ordered basis of $\Lambda_\phi$.

    Let's write $\zeta_\phi=\sum_i z'_i\otimes\lambda'_i$; if we denote by $({\lambda'_i}^*)_i$ the corresponding dual basis of $\hat\Lambda_\phi$, if we call $\psi$ Poonen's isomorphism $\hat\Lambda_\phi\cong\ker(\exp_\phi^*)\subseteq\C$, by the Remark \ref{oss universal dual special function} we know that $z'_i=\psi({\lambda'_i}^*)$; in particular, for $i\gg0$, we have \[z'_i=\phi_a^*(z'_{i+r\deg(a)}).\]

    Let's write $\phi_a^*=\sum_k \tau^{-k}a_k$. There is some real constant $\varepsilon>0$ such that that, for any $c\in\C$ with $\|c\|<\varepsilon$, \[\|\phi_a^*(c)\|=\|a_{r\deg(a)} c\|^{q^{-r\deg(a)}}<\|c\|.\] Since the sequence $(z'_i)_i$ converges to $0$, for $i\gg0$ we have \[\|z'_i\|=\|\phi_a^*(z'_{i+r\deg(a)})\|=\|a_{r\deg(a)} z'_{i+r\deg(a)}\|^{q^{-r\deg(a)}},\]
    hence
    \[\|z'_{i+r\deg(a)}\|=\|z'_i\|^{q^{r\deg(a)}}\|a_{r\deg(a)}\|^{-1}.\] By recursion, for all $k\geq0$ there is a positive real constant $\varepsilon<1$ and a positive integer $M$ such that, for $i>M$:
    \[\|z'_{i+kr\deg(a)}\|=\|z'_i\|^{q^{kr\deg(a)}}\|a_{r\deg(a)}\|^{-\frac{q^{kr\deg(a)}-1}{q-1}}<\varepsilon^{q^{kr\deg(a)}}\]
    In particular, by setting $i=M+1,\dots,M+r\deg(a)$, and setting $\delta\coloneqq\varepsilon^{q^{-M-r\deg(a)}}<1$, we deduce that $\|z'_n\|<\delta^{q^n}$ for $n\geq M$. In particular, the series $\sum_i {z'_i}^{q^k}\lambda'_i$ converges in $\C$ for any integer $k$.

    For all $i$, we can write $\lambda'_i$ as a finite sum $\sum_j\alpha_{i,j}\lambda_j$ with constants $\alpha_{i,j}\in\F_q$, so we have:
    \[\zeta_\phi=\sum_i z'_i\otimes\lambda'_i=\sum_i z'_i\otimes\left(\sum_j\alpha_{i,j}\lambda_j\right)=\sum_i\sum_j\alpha_{i,j} z'_i\otimes\lambda_j=\sum_j\left(\sum_i\alpha_{i,j }z'_i\right)\otimes\lambda_j.\]
    For all $j$, we deduce $z_j=\sum_i\alpha_{i,j}z'_i$. Moreover, for any integer $k$:
    \[\sum_j z_j^{q^k}\lambda_j=\sum_j\left(\sum_{i\geq j}\alpha_{i,j}{z'_i}^{q^k}\right)\lambda_j=\sum_i{z'_i}^{q^k}\left(\sum_{j\leq i}\alpha_{i,j}\lambda_j\right)=\sum_i {z'_i}^{q^k}\lambda'_i.\]
    For all $k$, let's set $l'_k\coloneqq\sum_i {z_i'}^{q^k}\lambda_i'$.
    If $k>0$, we have:
    \begin{align*}
        l_k&=\sum_{\lambda\in\Lambda_\phi}\lambda^{1-q^k}=\sum_{\lambda\in\Lambda_\phi}\lambda^{-q^k}\sum_i{\lambda'_i}^*(\lambda)\lambda'_i=\sum_i\left(\sum_{\lambda\in\Lambda_\phi}\lambda^{-q^k}{\lambda'_i}^*(\lambda)\right)\lambda'_i\\
        &=\sum_i\left(\sum_{\lambda\in\Lambda_\phi}\lambda^{-1}{\lambda'_i}^*(\lambda)\right)^{q^k}\lambda'_i=\sum_i g({\lambda'_i}^*)^{q^k}\lambda'_i=\sum_i {z'_i}^{q^k}\lambda_i'=l'_k.
    \end{align*}
    Note that for all $a\in A$, $\sum_i z'_i\otimes a\lambda_i=\sum_i\phi_a^*(z'_i)\otimes\lambda_i$, hence for any integer $k$ we have the identity:
    \[\sum_i {z'_i}^{q^k}\lambda_i=\sum_i\phi_a^*(z'_i)^{q^k}\lambda_i.\]
    Define $\log'_\phi\coloneqq\sum_k l'_k\tau^k\in\C[\![\tau^{-1},\tau]\!]$. For all $a\in A$, if we write $\phi_a^*=\sum_j\tau^{-j}a_j$, we have:
    \begin{align*}
        a\log'_\phi&=a\sum_kl'_k\tau^k=\sum_k\sum_i a\lambda'_i{z'_i}^{q^k}\tau^k 
        =\sum_k\sum_i\lambda'_i\phi_a^*(z'_i)^{q^k}\tau^k\\&=\sum_k\sum_i\lambda'_i\left(a_j z'_i \right)^{q^{k-j}}\tau^k
        =\sum_k\sum_j \left(\sum_i\lambda'_i\tau^{k-j}z'_i\right) a_j \tau^j\\&=\sum_k\sum_j l'_{k-j}\tau^{k-j}a_j\tau^j=\log'_\phi\circ\phi_a.
    \end{align*}
    Since $\log_\phi$ has the same property, $\log_\phi-\log'_\phi$ is a series in $\C[\![\tau^{-1}]\!]$ such that $a(\log_\phi-\log'_\phi)=(\log_\phi-\log'_\phi)\phi_a$ for all $a\in A$. Since the degrees of both sides differ if $a\not\in\F_q$ and $\log_\phi-\log'_\phi\neq0$, we deduce that $\log_\phi=\log'_\phi$, hence $l'_k=0$ for all $k<0$ and $l'_0=1$.
    
\end{proof}

\section{Pairing between Anderson eigenvectors and dual Anderson eigenvectors}\label{section pairing}

Let's fix a Drinfeld module $\phi$ with exponential $\exp_\phi=\sum_{i\geq0}e_i\tau^i$ and assume that $\Lambda_\phi:=\ker\exp_\phi\subseteq\C$ has rank $r$ as an $A$\=/module; let $\log_\phi=\sum_{i\geq0}l_i\tau^i$ be the inverse of $\exp_\phi$ as an element of $\C[\![\tau]\!]$.

\subsection{Definition of the dot product}

The following rationality result (a weak version of \cite[Thm. 6.3]{Ferraro}) links Anderson eigenvectors and dual Anderson eigenvectors in the rank 1 case.

\begin{teo}[Ferraro]\label{Ferraro}
Assume that $\phi$ is a normalized Drinfeld module of rank $1$. The product of an element in $\Sf_{\phi^*}(A)$ and an element in $\Sf_{\phi}(A)$ is a rational function over $X_\C$.
\end{teo}

To generalize this statement to Drinfeld modules of arbitrary rank, we need a proper way of "multiplying" $\zeta_\phi$ and $\omega_\phi$, established in the following lemma.

\begin{lemma}\label{ev}
The following $A_\C$\=/linear pairing is well defined:
\[\begin{tikzcd}[column sep=scriptsize, row sep=small]
&\C\hat\otimes\Lambda_\phi\arrow[r,phantom,"\otimes"] & \C\hat\otimes(\Lambda_\phi^*\otimes_A\Omega)\arrow[r]&\C\hat\otimes\Omega\\
&\sum_i c_i\otimes\lambda_i\arrow[r,phantom,"\otimes"] 
&\sum_j d_j\otimes(\lambda^*_j\otimes\omega_j)\arrow[r,phantom,"\mapsto"] 
&\sum_{i,j}(c_i d_j)\otimes(\lambda^*_j(\lambda_i)\omega_j)\\
&&f\arrow[u,phantom,sloped,"\coloneqq"]&g\arrow[u,phantom,sloped,"\coloneqq"]
\end{tikzcd}\]
Moreover, considering $g$ and $f$ as continuous functions respectively from ${\K}/{A}$ and $\Lambda_\phi\otimes_A{\K}/{A}$ to $\C$, for all $b\in{\K}/{A}$ we have:
\[g(b)=\sum_i c_i f(\lambda_i\otimes b).\]
\end{lemma}
\begin{proof}
The morphism is well defined because for all $\varepsilon>0$ there are finitely many pairs of indices $(i,j)$ such that $\|c_id_j\|>\varepsilon$; the $A_\C$\=/linearity is also obvious from the definition. Call $\res:\Omega\otimes\K/A\to\F_q$ and $\res_{\Lambda_\phi}:(\Lambda_\phi^*\otimes_A\Omega)\otimes\left(\Lambda_\phi\otimes_A{\K}/{A}\right)\to\F_q$ the two perfect pairings. By Remark \ref{oss dual} we have:
\[g(b)=\sum_{i,j}c_id_j\res(\lambda^*_j(\lambda_i)\omega_j,b)=\sum_i c_i\sum_j d_j\res_\Lambda(\lambda^*_j\otimes\omega_j,\lambda_i\otimes b)=\sum_i c_i f(\lambda_i\otimes b).\tag*{\qedhere}\]
\end{proof}

\subsection{Rationality of the dot products \texorpdfstring{$\zeta_\phi\cdot\omega_\phi^{(k)}$}{ζ · ω}}

The pairing defined in Lemma \ref{ev} will be denoted by a dot product. For any element $h\in\C\hat\otimes\Omega=\Hom^\mathrm{cont}_{\F_q}\left({\K}/{A},\C\right)$ and for any $b\in\K$ with projection $\closure{b}\in{\K}/{A}$, to simplify notation we will write $h(b)$ to denote $h(\closure{b})$. We can now state the generalization of Theorem \ref{Ferraro}.

\begin{teo}\label{Teo rationality}
For any Drinfeld module $\phi$, for all integers $k$, the dot product $\zeta_\phi\cdot\omega_\phi^{(k)}$ in $\C\hat\otimes\Omega$ is a rational differential form over the base-changed curve $X_{\C}$. Moreover, for all positive integers $k$, $\zeta_\phi\cdot\omega_\phi^{(k)}\in \Omega_\C$.
\end{teo}

\begin{proof}
    As an element of $\Hom_{\F_q}^\mathrm{cont}\left({\K\Lambda_\phi}/{\Lambda_\phi},\C\right)$, $\omega_\phi$ sends the projection of any $c\in\K\Lambda_\phi$ to $\exp_\phi(c)$. By Lemma \ref{ev}, since $\zeta_\phi=-\sum_{\lambda\in\Lambda_\phi\setminus\{0\}}\lambda^{-1}\otimes\lambda$, for all $b\in\K$ and for all integers $k$ we have:
    \[\zeta_\phi\cdot\omega_\phi^{(k)}(b)=-\sum_{\lambda\in\Lambda_\phi\setminus\{0\}}\frac{\exp(b\lambda)^{q^k}}{\lambda}=\left(-\sum_{\lambda\in\Lambda_\phi\setminus\{0\}}\frac{\exp(b\lambda)}{\lambda^{q^{-k}}}\right)^{q^k}.\]
    By Proposition \ref{identity 2}, for all positive integers $k$, if $b\in\K$ has norm at most $q^{-\frac{k+1}{r}}$, $\zeta_\phi\cdot\omega_\phi^{(k)}(b)=0$. Let's denote by $C\subseteq{\K}/{A}$ the subspace generated by the projections of elements in $\K$ with norm at most $q^h$, and denote by $Q$ the quotient. Since $Q$ is a finitely generated $\F_q$\=/vector space, we get the following:
    \[\Hom_{\F_q}^\mathrm{cont}\left(\faktor{\K}{A},\C\right)\supseteq\Hom_{\F_q}^\mathrm{cont}(Q,\C)=\Hom_{\F_q}(Q,\C)=\C\otimes\hat{Q}.\]
    Since $\zeta_\phi\cdot\omega_\phi^{(k)}$ restricted to $C$ is identically $0$, it's contained in $\C\otimes\hat{Q}$, therefore it can be expressed as a finite sum:
    \[\zeta_\phi\cdot\omega_\phi^{(k)}=\sum_i c_i\otimes\mu_i\in\C\otimes\hat{Q}\subseteq\C\otimes\widehat{\faktor{\K}{A}}=\Omega_\C.\]

    To prove the theorem for all integers $k$ we proceed by induction. Assume that the result holds for all integers bigger than $k$, and fix some $a\in A\setminus\F_q$. From the definition of special functions we have:
    \begin{align*}
        &(1\otimes a-a\otimes1)\omega_\phi=\sum_{i=1}^{r\deg(a)}(\phi_a)_i\omega_\phi^{(i)}\\
        \Longrightarrow&\zeta_\phi\cdot\omega_\phi^{(k)}=\frac{1}{1\otimes a-a^{q^k}\otimes 1}\sum_{i=1}^{r\deg(a)}(\phi_a)_i^{q^k}\zeta_\phi\cdot\omega_\phi^{(k+i)},
    \end{align*}
    hence $\zeta_\phi\cdot\omega_\phi^{(k)}$ is a rational differential form over $X_\C$.
\end{proof}

\begin{oss}\label{oss computation}
    From the previous proof we deduce that, if we can compute the dot product $\zeta_\phi\cdot\omega_\phi^{(k)}$ for $r\deg(a)$ consecutive integers $k$, then we can compute it for any value of $k$.
\end{oss}

\subsection{Computation of the dot products \texorpdfstring{$\zeta_\phi\cdot\omega_\phi^{(k)}$  for $k\ll0$}{ω · ζ}}

We can expand on the previous theorem. In fact, we are able to describe explicitly the differential form $\zeta_\phi^{(k)}\cdot\omega_\phi$ for $k$ large enough by using once again Proposition \ref{identity 2}.

\begin{teo}\label{Teo pairing for large k}
    For all $b\in\K$ denote by $s(b)\in\K$ an element of smallest norm such that $b-s(b)\in A$. For all integers $k>re\left(\left\lfloor\frac{2g-2}{e}\right\rfloor+1\right)$, we have the following identity for all $b\in\K$:
    \[\zeta_\phi^{(k)}\cdot\omega_\phi(b)=\sum_{j=0}^k e_jl_{k-j}^{q^j}s(b)^{q^j}.\]
\end{teo}

\begin{proof}
    Recall that the norm of all elements in $\K$ is an integer power of $q^e$. Fix any $b\in\K$, assume $\|s(b)\|=q^{ed}$ for some integer $d$; the $\F_q$\=/vector space ${H^0(X,d\infty)}/{H^0(X,(d-1)\infty)}$ has dimension less than $e$, otherwise there would be some $a\in H^0(X,d\infty)\subseteq A$ such that $\|s(b)-a\|<\|s(b)\|$, contradicting the minimality condition on $s(b)$. By Riemann--Roch, if $e(d-1)>2g-2$, the spaces $H^0(X,(d-1)\infty)$ and $H^0(X,d\infty)$ have dimension respectively $e(d-1)-g+1$ and $ed-g+1$, which is a contradiction, hence $\|s(b)\|\leq q^{e\left(\left\lfloor\frac{2g-2}{e}\right\rfloor+1\right)}$. Since $\frac{k-1}{r}\geq e\left(\left\lfloor\frac{2g-2}{e}\right\rfloor+1\right)$, by Proposition \ref{identity 2} we have:
    \[\zeta_\phi^{(k)}\cdot\omega_\phi(b)=\zeta_\phi^{(k)}\cdot\omega_\phi(s(b))=-\sum_{\lambda\in\Lambda_\phi\setminus\{0\}}\frac{\exp_\phi(s(b)\lambda)}{\lambda^{q^k}}=\sum_{j=0}^k e_jl_{k-j}^{q^j}s(b)^{q^j}.\tag*{\qedhere}\]
\end{proof}
\begin{oss}\label{oss pairing for large k}
    Equivalently, for all integers $k>re\left(\left\lfloor\frac{2g-2}{e}\right\rfloor+1\right)$ and for all $b\in\K$:
    \[\zeta_\phi\cdot\omega_\phi^{(-k)}(b)=\left(\sum_{j=0}^k e_jl_{k-j}^{q^j}s(b)^{q^j}\right)^{q^{-k}}.\]
    In principle, we can use this result to compute $\zeta_\phi\cdot\omega_\phi^{(i)}(b)$ for all $i$ and all $b$, in the same way we proved rationality in Theorem \ref{Teo rationality}, as we observed in Remark \ref{oss computation}.
\end{oss}

\subsection{The generating series of the dot products \texorpdfstring{$\zeta_\phi\cdot\omega_\phi^{(k)}$}{ζ · ω}}

Using Theorem \ref{Teo pairing for large k} and Remark \ref{oss pairing for large k}, we can in principle compute the dot product $\zeta_\phi\cdot\omega_\phi^{(k)}$ for any $k\geq-re\left(\left\lfloor\frac{2g-2}{e}\right\rfloor+1\right)$, but since the sketched algorithm is recursive, it's necessary to compute all the intermediate dot products $\zeta_\phi\cdot\omega_\phi^{(i)}$ for $i$ between $-re\left(\left\lfloor\frac{2g-2}{e}\right\rfloor+1\right)$ and $k$.

The objective of this subsection is to streamline this computation by studying the generating series $\sum_{k\in\Z}\zeta_\phi\cdot\omega_\phi^{(k)}\tau^k$.

\begin{Def}
Denote by $\C\langle\tau\rangle$ the subset of $\C[\![\tau]\!][\tau^{-1}]$ given by the series with a nonzero radius of convergence on $\C$.
\end{Def}
\begin{oss}
    The set $\C\langle\tau\rangle$ is closed under addition and composition, hence it is a subring of $\C[\![\tau]\!][\tau^{-1}]$.
\end{oss}
\begin{oss}\label{oss C<t>}
    Since the radius of convergence of $h=\sum_i h_i\tau^i\in\C[\![\tau]\!][\tau^{-1}]$ is the inverse of $\limsup_{i\to\infty}\|h_i\|^{q^{-i}}$, we have that $h\in\C\langle\tau\rangle$ if and only if $\limsup_{i\to\infty}\|h_i\|^{q^{-i}}<\infty$.
\end{oss}

\begin{lemma}
     Every nonzero element $h\in\C[\tau,\tau^{-1}]$ admits a (unique) bilateral inverse in $\C\langle\tau\rangle$.    
\end{lemma}
\begin{proof}
    Since $\tau:\C\to\C$ is an isomorphism, up to multiplication we can assume \[h=\sum_{i\geq0}h_i\tau^i,\]
    with $h_0=1$. If we call 
    \[h_+\coloneqq-\sum_{i\geq1}h_i\tau^i=1-h,\]
    the series $\sum_{i\geq0}h_+^i$ is a well defined bilateral inverse of $h$ in $\C[\![\tau]\!]$. Since $h$ has finitely many nonzero coefficients, it's easy to see that there is some $R\in\R_{>0}$ and some positive real constant $C<1$ such that, for all $x\in\C$ with norm less than $R$, $\|h_ix^{q^i}\|\leq C\|x\|$ for all $i\geq1$. In particular, for all $x\in\C$ with norm less than $R$, each of the finitely many summands in the expansion of $h_+^i(x)$ has norm at most $C^i\|x\|$, hence the series \[\sum_{i\geq0}h_+^i(x)\] converges. We deduce that the series $\sum_{i\geq0}h_+^i$ has a nonzero radius of convergence, hence it belongs to $\C\langle\tau\rangle$.
\end{proof}

\begin{Def}
    For all $c\in\K$ we define $\Phi_c\in\C\langle\tau\rangle$ as $\exp_\phi\circ c\circ\log_\phi$.
\end{Def}

\begin{oss}
    For all $a\in A$, $\Phi_a=\phi_a$. The map $\Phi:\K\to\C\langle\tau\rangle$ sending $c$ to $\Phi_c$ is the unique ring homomorphism which extends $\phi:A\to\C\langle\tau\rangle$ such that each coefficient is a continuous function.
\end{oss}

For a series $s\in\C[\![\tau,\tau^{-1}]\!]$ and an integer $k$, let's denote by $(s)_k$ its $k$-th coefficient, so that $s=\sum_{k\in\Z}(s)_k\tau^k$.

\begin{prop} 
    Let $\mu:\K\to\C[\![\tau,\tau^{-1}]\!]$ be a function with the following properties:
    \begin{enumerate}
        \item\label{(i)} $\forall k\in\Z$ the function sending $c$ to $(\mu_c)_k$ is $\F_q$\=/linear and continuous;
        \item\label{(ii)} $\forall a\in A,c\in\K$, $\mu_{ac}=\mu_c\phi_a$;
        \item\label{(iii)} $\forall a\in A$, $\mu_a=0$;
        \item\label{(iv)} $\forall R\in\R$ there is some $n_0\in\Z$ such that for all $n\geq n_0$, for all $c\in\K$ with $\|c\|\leq R$, $(\mu_c)_n=(\Phi_c)_n$.
    \end{enumerate}
    The function $\mu$ is uniquely determined; in addition, for any $c\in\K$, we have:
    \[\mu_c=\sum_{k\in\mathbb\Z} \left(\zeta_\phi^{(k)}\cdot\omega_\phi\right)(c)\tau^k.\]
\end{prop}
\begin{proof}
    To prove uniqueness, let's take two such functions $\mu$ and $\mu'$, and define \[\lambda:=\mu-\mu'.\] For each element $c\in\K$ let $s(c)$ be an element of least norm such that $c-s(c)\in A$. As we already said in the proof of Theorem \ref{Teo pairing for large k}, for all $c\in\K$, \[\|s(c)\|\leq q^{e\left(\left\lfloor\frac{2g-2}{e}\right\rfloor+1\right)};\] using properties \ref{(i)},\ref{(iii)}, and \ref{(iv)} with $R=q^{e\left(\left\lfloor\frac{2g-2}{e}\right\rfloor+1\right)}$, we deduce that there is some integer $n_0$ such that, for all $n\geq n_0$, for all $c\in\K$:
    \[(\lambda_c)_n=(\lambda_{s(c)})_n+(\lambda_{c-s(c)})_n=(\lambda_{s(c)})_n=(\mu_{s(c)})_n-(\mu'_{s(c)})_n=(\Phi_{s(c)})_n-(\Phi_{s(c)})_n=0.\]
    If by contradiction $\lambda\not\equiv0$, there is an element $c\in\K$ such that $\lambda_c$ has the highest degree; by property \ref{(ii)}, for any $a\in A\setminus\F_q$, $\lambda_{ac}=\lambda_c\phi_a$, which has a greater degree than $\lambda_c$, reaching a contradiction.

    Let's check that \[\mu_c\coloneqq\sum_{k\in\mathbb\Z} \left(\zeta_\phi^{(k)}\cdot\omega_\phi\right)(c)\tau^k\] satisfies all conditions. The properties \ref{(i)} and \ref{(iii)} are obvious. For property \ref{(iv)}, note that for all $c\in\K$
    \[(\Phi_c)_k=(\exp_\phi\circ c\circ\log_\phi)_k=\sum_{i+j=k}e_ic^{q^i}l_j^{q^i},\]
    which is equal to $\left(\zeta_\phi^{(k)}\cdot\omega_\phi\right)(c)$ for all $k\geq r\cdot\log_q(\|c\|)+1$ by Proposition \ref{identity 2}. Finally, for property \ref{(ii)}, since $\zeta_\phi$ is an Anderson eigenvector, for all $a\in A$ we have:
    \[(1\otimes a)\zeta_\phi=\sum_{i=0}^{r\deg(a)}(\phi_a)_i^{q^{-i}}\zeta_\phi^{(-i)}\Rightarrow\text{for all integers $k$, }(1\otimes a)\zeta_\phi^{(k)}=\sum_{i=0}^{r\deg(a)}(\phi_a)_i^{q^{k-i}}\zeta_\phi^{(k-i)}.\]
    We deduce that, for all $c\in\K$:
    \begin{align*}
        \mu_c\phi_a&= \left(\sum_{k\in\Z}\left(\zeta_\phi^{(k)}\cdot\omega_\phi\right)(c)\tau^k\right)\left(\sum_{i=0}^{r\deg(a)}(\phi_a)_i\tau^i\right)\\
        &=\sum_{k\in\Z}\left(\sum_{i=0}^{r\deg(a)}(\phi_a)_i^{q^{k-i}}\left(\zeta_\phi^{(k-i)}\cdot\omega_\phi\right)(c)\right)\tau^k\\
        &=\sum_{k\in\Z}\left((1\otimes a)\zeta_\phi^{(k)}\cdot\omega_\phi\right)(c)\tau^k=\sum_{k\in\Z}\left(\zeta_\phi^{(k)}\cdot\omega_\phi\right)(ac)\tau^k=\mu_{ac}.\tag*{\qedhere}
    \end{align*}
\end{proof}
\begin{oss}
    If $c\in\K\setminus A$, the formal power series $\mu_c$ has no obvious convergence properties. In fact, if $c\in K\setminus A$ we can choose $a\in A$ so that $ca\in A$, and we get that $\mu_c\phi_a=\mu_{ac}=0$: since $\mu_c\neq0$, this implies that its radius of convergence is $0$.
\end{oss}

\begin{Def}
    For all $c\in\K$ we define $\hat{\Phi}_c\coloneqq(\Phi_c-\mu_c)^*\in\C[\![\tau,\tau^{-1}]\!]$.
\end{Def}

\begin{prop}
    For all $c\in\K$, the series $\hat{\Phi}_c$ has a nonzero radius of convergence. Moreover, the map $\hat{\Phi}:\K\to\C\langle\tau\rangle$ sending $c$ to $\hat{\Phi}_c$ is the unique ring homomorphism which extends $\phi^*:A\to\C\langle\tau\rangle$ such that each coefficient is a continuous function.
\end{prop}
\begin{proof}
    Uniqueness is obvious: by multiplicativity there is at most one way to extend $\phi^*$ to the fraction field $K$, and by continuity there is at most one way to extend it to the completion $\K$. By definition of $\Phi$ and $\mu$, each coefficient of $\hat{\Phi}_c$ is a continuous function of $c$. 
    
    For all $c\in\K$, by Proposition \ref{identity 2} we have $(\hat{\Phi}_c)_k=\left((\Phi_c-\mu_c)^*\right)_k=0$ for $k\ll0$, hence $\hat{\Phi}_c\in\C[\![\tau]\!][\tau^{-1}]$.
    On the other hand, for $k\gg0$:
    \[((\hat{\Phi}_c)_k)^{q^{-k}}=((-\mu_c^*)_k)^{q^{-k}}=-(\mu_c)_{-k}=-\left(\zeta_\phi^{(-k)}\cdot\omega_\phi\right)(c)=\sum_{\lambda\in\Lambda_\phi\setminus\{0\}}\frac{\exp(c\lambda)}{\lambda^{q^{-k}}};\]
    all the numerators of the series belong to the compact space $\exp(\K\Lambda_\phi)\cong{\K\Lambda_\phi}/{\Lambda_\phi}$, and since $\Lambda_\phi\subseteq\C$ is discrete all the denominators are bounded from below: this means that the set $\{((\hat{\Phi}_c)_k)^{q^{-k}}\}_{k\gg0}$ is bounded, hence $\hat{\Phi}_c\in\C\langle\tau\rangle$ by Remark \ref{oss C<t>}.
    For all $a\in A$, for all $c\in \K$:
    \begin{align*}
    \hat{\Phi}_a&=(\Phi_a-\mu_a)^*=\phi^*_a\\
    \phi^*_a\circ\hat{\Phi}_c&=(\Phi_c\circ\phi_a-\mu_c\circ\phi_a)^*=(\Phi_{ac}-\mu_{ac})^*=\hat{\Phi}_{ac},
    \end{align*}
    which implies that $\hat{\Phi}$ extends $\phi^*$ multiplicatively.
\end{proof}

\begin{oss}
    For all $c\in\K$ we have:
    \[\mu_c^*=\sum_{k\in\Z}\left(\zeta_\phi\cdot\omega_\phi^{(k)}\right)(c)\tau^k.\]
\end{oss}
A posteriori, we can repackage the results of this subsection under the following theorem.

\begin{teo}\label{main teo}
    Let $\Phi,\hat{\Phi}:\K\to\C\langle\tau\rangle$ be the unique ring homomorphisms which extend respectively $\phi,\phi^*:A\to\C\langle\tau\rangle$ and such that their $k$\=/th coefficient is a continuous function from $\K$ to $\C$ for all $k\in\Z$. The following identity holds in the $\C[\tau,\tau^{-1}]$\=/module $\C[\![\tau,\tau^{-1}]\!]$ for all $c\in\K$:
    \[\sum_{k\in\mathbb\Z} \left(\zeta_\phi\cdot\omega_\phi^{(k)}\right)(c)\tau^k=\Phi_c^*-\hat{\Phi}_c.\]
\end{teo}

This Theorem allows us to partially carry out the computation of the dot products $\zeta_\phi\cdot\omega_\phi^{(k)}$, such as in the following Proposition.

\begin{prop}\label{cor 0 forms}
    For all $c\in\K$ with norm less than $1$:
    \[\left(\zeta_\phi\cdot\omega_\phi^{(k)}\right)(c)=\begin{dcases}
        c&\text{if }k=0\\
        0&\text{if }1\leq k\leq r-1.
    \end{dcases}\]
\end{prop}
\begin{proof}
    For all $c\in\K$ the lowest degree of $\hat\Phi_c$ is $-r\deg(c)$, while the highest degree of $\Phi_c^*$ is $0$. In particular, if $\|c\|<1$, i.e. $\deg(c)\leq-1$, we have:
    \[\left(\zeta_\phi\cdot\omega_\phi^{(k)}\right)(c)=(\Phi^*_c-\hat{\Phi}_c)_k=\begin{dcases}
        (\Phi^*_c-\hat{\Phi}_c)_0=(\Phi^*_c)_0=c&\text{if }k=0\\
        (\Phi^*_c-\hat{\Phi}_c)_k=0&\text{if }1\leq k\leq r-1.\tag*{\qedhere}
    \end{dcases}\]
    \end{proof}

\subsection{Application to the case of genus \texorpdfstring{$0$}{0} and arbitrary rank}

Thanks to Theorem \ref{main teo}, we can compute efficiently the dot products $\zeta_\phi\cdot\omega_\phi^{(k)}$ in the case of genus $0$ and rational point at infinity. In this subsection we assume $X=\mathbb{P}^1_{\F_q}$, and we fix a rational function $\theta$ over $X$ with a simple pole at $\infty$. In this case we can write $A=\F_q[\theta]$, $\K=\F_q(\!(\theta^{-1})\!)$ and $\Omega=\F_q[\theta]d\theta$, where $d\theta:{\K}/{A}\to\F_q$ sends $\theta^n$ to $\delta_{-1,n}$ for all $n\in\Z$.

\begin{prop}\label{g=0 i=0}
Let $\phi:\F_q[\theta]\to\C[\tau]$ be a Drinfeld module of rank $r$. We have the following identities in $\C\hat\otimes\Omega$:
    \begin{align*}
        \zeta_\phi\cdot\omega_\phi&=\frac{d\theta}{\theta\otimes1-1\otimes\theta};\\
        \zeta_\phi\cdot\omega_\phi^{(k)}&=0\;\;\;\forall1\leq k\leq r-1.
    \end{align*}
\end{prop}

\begin{proof}
By Proposition \ref{cor 0 forms}, for all $n>0$ we have:
\[\left(\zeta_\phi\cdot\omega_\phi^{(k)}\right)(\theta^{-n})=\begin{dcases}
    \theta^{-n}&\text{if }k=0\\
    0&\text{if }1\leq k\leq r-1.
\end{dcases}\]
Since $\theta^n\in A$ for all $n\geq0$, we also have $\left(\zeta_\phi\cdot\omega_\phi^{(k)}\right)(\theta^n)=0$ for all $n\geq0$ and for all $k$ so, if $1\leq k\leq r-1$, $\zeta_\phi\cdot\omega_\phi^{(k)}$ is identically zero. If instead $k=0$ we have the following identity for all integers $n$:
\[\left((\theta\otimes1-1\otimes \theta)\zeta_\phi\cdot\omega_\phi\right)(\theta^n)=\theta\left(\zeta_\phi\cdot\omega_\phi(\theta^n)\right)-\zeta_\phi\cdot\omega_\phi(\theta^{n+1})=\delta_{-1,n}=d\theta(\theta^n),\]
hence \[\zeta_\phi\cdot\omega_\phi=\frac{d\theta}{(\theta\otimes1-1\otimes \theta)}.\tag*{\qedhere}\]
\end{proof}

We now relate the usual definition of Anderson generating functions to the universal Anderson eigenvector, by giving a basis-dependent description of the latter.

\begin{lemma}
    Fix the $A$\=/linear bases $\{\pi_1,\dots,\pi_r\}$ of $\Lambda_\phi$ and $\{\pi_1^*,\dots,\pi_r^*\}$ of $\Lambda_\phi^*$, where $\pi_i^*(\pi_k)=\delta_{i,k}$. Then, we have:
    \begin{align*}
    &\omega_\phi= \sum_{i=1}^r\sum_{j\geq0} \exp_\phi\left(\frac{\pi_i}{\theta^{j+1}}\right) \otimes\theta^j\pi_i^*d\theta,
    &\zeta_\phi=\sum_{i=1}^r\sum_{j\geq0}\left(\sum_{\lambda\in\Lambda\setminus\{0\}}\frac{d\theta\pi_i^*}{\theta^{j+1}}(\lambda)\lambda^{-1}\right)\otimes\theta^j\pi_i.
\end{align*}
\end{lemma}

\begin{proof}
When used as indices, we imply $i$ to vary among the integers between $1$ and $r$, extremes included, and $j$ to vary among the nonnegative integers.
The chosen bases induce an isomorphism \[\Hom_A(\Lambda_\phi,\Omega)\cong\bigoplus_i A d\theta\pi_i^*.\] The $\F_q$\=/linear basis $\{\theta^j d\theta\pi_i^*\}_{i,j}$ of $\Hom_A(\Lambda_\phi,\Omega)$ induces a dual basis $\{\theta^{-j-1}\pi_i\}_{i,j}$ of \[\widehat{\Hom_A(\Lambda_\phi,\Omega)}\cong\faktor{\K\Lambda_\phi}{\Lambda_\phi}.\] Similarly, the $\F_q$\=/linear basis $\{\theta^j\pi_i\}_{i,j}$ of $\Lambda_\phi$: induces the dual basis $\{\theta^{-j-1}d\theta\pi_i^*\}_{i,j}$ of \[\hat\Lambda_\phi\cong\faktor{\K\Hom_A(\Lambda_\phi,\Omega)}{\Hom_A(\Lambda_\phi,\Omega)}.\] This concludes the proof, by virtue of Remark \ref{oss universal special function} and the proof of Theorem \ref{zeta function functor}.
\end{proof}

\begin{Def}
For $i=1,\dots,r$ we define the $i$\=/th Anderson generating function as:
\[\omega_{\phi,i}\coloneqq\sum_{j\geq0}\exp_\phi\left(\frac{\pi_i}{\theta^{j+1}}\right)\otimes\theta^j\in\C\hat\otimes A.\]

Similarly, for $i=1,\dots,r$ we define the $i$\=/th dual Anderson generating function as:
    \[\zeta_{\phi,i}= \sum_{j\geq0} \left(\sum_{\lambda\in\Lambda\setminus\{0\}} \frac{d\theta\pi_i^*}{\theta^{j+1}}(\lambda)\lambda^{-1}\right)\otimes\theta^j\in\C\hat\otimes A.\]
\end{Def}
\begin{oss}
For all integers $1\leq i\leq r$, $\omega_{\phi,i}$ and $\zeta_{\phi,i}$ are the unique elements in $\C\hat\otimes A$ such that the identities $(1\otimes\pi_i)(\omega_\phi)=\omega_{\phi,i} d\theta$ and $(1\otimes\pi^*_i)(\zeta_\phi)=\zeta_{\phi,i}$ hold (in $\C\hat\otimes\Omega$ and $\C\hat\otimes A$, respectively).
\end{oss}

\begin{Def}
    Let's define $\boldsymbol{\omega}_\phi\coloneqq(\omega_{\phi,i}^{(j-1)})_{i,j}\in\Mat_{r\times r}(\C\hat\otimes A)$. We call it the \textit{rigid analytic trivialization} of the $t$\=/motive attached to $\phi$.
\end{Def}

The previous matrix has been studied in various articles (see for example \cite[Section 4.2]{Pellarin08}, \cite{KP23}, \cite{GP19}). We can use it to state the following Theorem.
\begin{teo}
    The product of $\zeta_\phi\in \Mat_{1\times r}(\C\hat\otimes A)$ and $\boldsymbol{\omega}_\phi\in\Mat_{r\times r}(\C\hat\otimes A)$ is the vector \[\frac{1}{(\theta\otimes1-1\otimes \theta)}\cdot(1,0,\dots,0)\in\Mat_{1\times r}(\C\hat\otimes A).\]
\end{teo}
\begin{proof}
    Note that we have interpreted $\zeta_\phi$ as $(\zeta_{\phi,i})_i\in \Mat_{1\times r}(\C\hat\otimes A)$. If we multiply by $d\theta\in\Omega$ the $j$\=/th coordinate of the product, we get:
    \[\sum_{i=1}^r\omega_{\phi,i}^{(j-1)} \zeta_{\phi,i}d\theta= \left(\sum_{i=1}^r\omega_{\phi,i}\pi_i^*d\theta\right)^{(j-1)}\cdot\left(\sum_{i=1}^r\zeta_{\phi,i}\pi_i\right)=\omega_\phi^{(j-1)}\cdot\zeta_\phi,\]
    which is $\frac{d\theta}{(\theta\otimes1-1\otimes \theta)}$ if $j=1$ by Proposition \ref{g=0 i=0}, and $0$ otherwise by Proposition \ref{cor 0 forms}.
\end{proof}

\begin{oss}
    It's a well known result 
    that the determinant of the matrix $\boldsymbol{\omega}_\phi$ is nonzero (see for example \cite[Prop. 6.2.4]{GP19}), so by the previous theorem we can recover $\zeta_\phi$ from $\boldsymbol{\omega}_\phi$.
\end{oss}

\subsection{Application to the case of hyperelliptic curves}

In the case of rank $1$ normalized Drinfeld modules, the result \cite[Thm. 7.26]{Ferraro} can be used to express the rational form $\zeta_\phi\cdot\omega_\phi$ in terms of the Drinfeld divisor. While Theorem $\ref{main teo}$, in principle, completely describes the form $\zeta_\phi\cdot\omega_\phi$, it's not as explicit a result for arbitrary curves.

In this subsection we restrict ourselves to the case of a hyperelliptic curve $X$ with hyperelliptic divisor $2\infty$ and a Drinfeld module $\phi$ of rank $1$. We use the results of the previous sections to recover an expression for the scalar product $\zeta_\phi\cdot\omega_\phi$ and for the shtuka function $f_\phi$ in terms of the coefficients of $\phi$. 

A curve $X$ of genus $g$ is hyperelliptic if and only if there is a divisor $D$ of degree $2$, called hyperelliptic divisor, such that $\dim_{\F_q}(H^0(X,\O_X(D)))=2$. If we assume $D=2\infty$, there is a rational function $x\in A$ of degree $2$. Let's denote by $y$ an element of $A$ with the smallest odd degree.

\begin{oss}
    An $\F_q$\=/linear basis of $A$ is $\mathcal{B}_0:=\{x^i,x^iy\}_{i\geq0}$. In particular, the only positive integers that are not degrees of elements in $A$ are the odd positive integers smaller than $\deg(y)$; by Riemann--Roch's theorem, this set has cardinality $g$, hence $\deg(y)=2g+1$. Expanding $y^2$ in terms of the basis $\mathcal{B}_0$, we deduce that there are polynomials $P,Q\in\F_q[t]$ such that $y^2=Q(x)y+P(x)$, where $P$ has degree $2g+1$ and $Q$ has degree at most $g$.

    If the characteristic of the base field is odd, we can also assume $Q(x)=0$ using the coordinate change $y\mapsto y+\frac{Q(x)}{2}$.
\end{oss}

\begin{oss}
    Every element of ${\K}/{A}$ can be represented by an element of $\K$ with degree either negative or equal to an odd positive number smaller than $2g+1$. We deduce that the image of $\mathcal{B}:=\{yx^{-i-1},x^{-i-1}\}_{i\geq0}$ in ${\K}/{A}$ is a set of linearly independent elements which spans a dense subset of ${\K}/{A}$.
\end{oss}

\begin{prop}
    If we define $\nu\in\Omega=\Hom^\mathrm{cont}_{\F_q}\left({\K}/{A},\F_q\right)$ as the function sending $yx^{-1}$ to $1$ and all the other elements of $\mathcal{B}$ to $0$, we get that $\Omega=A\nu$.
\end{prop}
\begin{proof}
    For all $j\geq0$, for all $c\in\K$, $(x^j\nu)(c)=\nu(x^jc)$, which is $1$ when $c=yx^{-j-1}$ and $0$ on all the other elements of $\mathcal{B}$. 
    
    Similarly, For all $j\geq0$, for all $c\in\K$, $((y-Q(x))x^j\nu)(c)=\nu((y-Q(x))x^jc)$. If $c=x^{-1-i}$ for some $i\geq0$ we have:
    \[((y-Q(x))x^j\nu)(c)=\nu(yx^{j-i-1})-\nu(Q(x)x^{j-i-1})=\nu(yx^{j-i-1})=\delta_{j,i}.\]
    If $c=yx^{-1-i}$ for some $i\geq0$ we have:
    \[((y-Q(x))x^j\nu)(c)=\nu((y^2-Q(x)y)x^{j-i-1})=\nu(P(x)x^{j-i-1})=0.\]

    In particular, the elements \[\{(y-Q(x))x^i\nu,x^i\nu\}_{i\geq0}\subseteq\Omega=\Hom_{\F_q}\left(\faktor{\K}{A},\F_q\right)\] are independent, and since $\mathcal{B}$ spans a dense subset of ${\K}/{A}$, they also generate all of $\Omega$.
\end{proof}

\begin{lemma}\label{lemma 1 g=1}
    Denote by $\left({\K}/{A}\right)_{<q^{-2}}\subseteq{\K}/{A}$ the subspace of the elements with norm less than $q^{-2}$, and call $C$ the cokernel of this inclusion. Then, the image of \[\{yx^{-i-1}\}_{0\leq i\leq g}\cup\{x^{-1}\}\] forms a basis of $C$, and the set \[\{x^i\nu\}_{0\leq i\leq g}\cup\{(y-Q(x))\nu\}\] is the corresponding dual basis of $\Hom_{\F_q}(C,\F_q)\subseteq\Omega$.
\end{lemma}
\begin{proof}
    On one hand, the images of $\{yx^{-i-1}\}_{0\leq i\leq g}\cup\{x^{-1}\}$ span $C$ because they are the only elements of $\mathcal{B}$ that are not sent to $0$ under the induced map $\K\to C$. On the other hand, 
    \begin{align*}
        \deg(yx^{-i-1})&=2(g-i)-1\text{ { } for }0\leq i\leq g,\\
        \deg(x^{-1})&=-2,
    \end{align*} hence their images are $\F_q$\=/linearly independent in $C$.
    
    Note that the image of $\{yx^{-j-1-g},x^{-j-1}\}_{j\geq1}$ in ${\K}/{A}$ spans a dense subset of $\left({\K}/{A}\right)_{<q^{-2}}$. For all $0\leq i\leq g$, for all $j\geq1$ we get:
    \[\begin{dcases}
    (x^i\nu)(yx^{-j-1-g})=\nu(yx^{i-j-1-g})=0&\text{because $i-j-1-g\leq-2$}\\
    (x^i\nu)(x^{-j-1})=\nu(x^{i-j})=0&
    \end{dcases}\]
    and
    \[\begin{dcases}
    ((y-Q(x))\nu)(yx^{-j-1-g})=\nu(P(x)x^{-j-1-g})=0&\\
    ((y-Q(x))\nu)(x^{-j-1})=\nu(yx^{-j-1})-\nu(Q(x)x^{-j-1})=0&\text{because $-j-1\leq-2$},
    \end{dcases}
    \]
    so $\{x^i\nu\}_{0\leq i\leq g}\cup\{y\nu\}\in\Hom_{\F_q}(C,\F_q)$. On the other hand, we have the following identities for all $0\leq i\leq g$ and for all $0\leq j\leq g$:
    \[\begin{dcases}
    (x^i\nu)(yx^{-j-1})=\nu(yx^{i-j-1})=\delta_{i,j}\\
    (x^i\nu)(x^{-1})=\nu(x^{i-1})=0
    \end{dcases}\]
    and
    \[\begin{dcases}
    ((y-Q(x))\nu)(yx^{-j-1})=\nu(P(x)x^{-j-1})=0\\
    ((y-Q(x))\nu)(x^{-1})=\nu(yx^{-1})-\nu(Q(x)x^{-1})=1.
    \end{dcases}
    \]
    This implies that \[\{x^i\nu\}_{0\leq i\leq g}\cup\{(y-Q(x))\nu\}\] is the dual basis of \[\{yx^{-i-1}\}_{0\leq i\leq g}\cup\{x^{-1}\},\] as desired.
\end{proof}

\begin{oss}
By Theorem \ref{main teo}, we have the following identity for all $c\in K$ and for all $i\in\Z$:
    \[\left(\zeta_\phi\cdot\omega^{(i)}_\phi\right)(c)=((\phi_c)^*-(\phi^*)_c)_i,\]
    where $\left(\zeta_\phi\cdot\omega^{(i)}_\phi\right)$ is considered as a continuous homomorphism from ${\K}/{A}$ to $\C$. In particular, for all $c\in K$ with degree less than $-i$:
    \[\left(\zeta_\phi\cdot\omega^{(i)}_\phi\right)(c)=\begin{cases*}
        c\text{ if }i=0\\
        0\text{ if }i>0.
    \end{cases*}\] 
Moreover, for all $0\leq i\leq g$ we have:
    \begin{align*}
        &(\zeta_\phi\cdot\omega_\phi)(yx^{-i})=yx^{-i}-\left(\phi^*_{yx^{-i}}\right)_0;\\
    &\left(\zeta_\phi\cdot\omega^{(1)}_\phi\right)(yx^{-i})=-\left(\phi^*_{yx^{-i}}\right)_1.
    \end{align*}
\end{oss}
 
\begin{teo}\label{prop hyper}
    We have the following identities for the dot product $\zeta_\phi\cdot\omega_\phi$ and the shtuka function $f_\phi$:
    \begin{align*}
    \zeta_\phi\cdot\omega_\phi&=\left(\frac{y\otimes 1+1\otimes (y-Q(x))}{x\otimes1-1\otimes x}-\sum_{i=0}^{g-1}\left(\phi^*_{yx^{-i-1}}\right)_0\otimes x^i\right)(1\otimes \nu)\\
        f_\phi&=\frac{(x\otimes1-1\otimes x)\left(-\sum_{i=0}^g \left(\phi^*_{yx^{-i-1}}\right)_1\otimes x^i\right)}{y\otimes 1+1\otimes (y-Q(x))-(x\otimes1-1\otimes x)\left(\sum_{i=0}^{g-1}\left(\phi^*_{yx^{-i-1}}\right)_0\otimes x^i\right)}.
    \end{align*}
\end{teo}
\begin{proof}
    For all $c\in\K$ of norm less than $1$, $(\zeta_\phi\cdot\omega_\phi)(c)=c$. In particular, for all $c\in\left({\K}/{A}\right)_{<q^2}$ we have:
    \[(x\otimes1-1\otimes x)(\zeta_\phi\cdot\omega_\phi)(c)=x(\zeta_\phi\cdot\omega_\phi)(c)-(\zeta_\phi\cdot\omega_\phi)(xc)=0.\]
In particular, by Lemma \ref{lemma 1 g=1} $(x\otimes1-1\otimes x)(\zeta_\phi\cdot\omega_\phi)$ is completely determined by its evaluation at $\{yx^{-i-1}\}_{0\leq i\leq g}\cup\{x^{-1}\}$ as a function from ${\K}/{A}$ to $\C$. Since $(\zeta_\phi\cdot\omega_\phi)(yx^{-i})=yx^{-i}-\left((\phi^*_x)^{-i}\circ\phi^*_y\right)_0$ for all $0\leq i\leq g$, we can compute the following evaluations:
\begin{align*}
    (x\otimes1-1\otimes x)(\zeta_\phi\cdot\omega_\phi)(yx^{-i-1})&=x(\zeta_\phi\cdot\omega_\phi)(yx^{-i-1})-(\zeta_\phi\cdot\omega_\phi)(yx^{-i})\\
    &=\left(\phi^*_{yx^{-i}}-x\phi^*_{yx^{-i-1}}\right)_0;\\
(x\otimes1-1\otimes x)(\zeta_\phi\cdot\omega_\phi)(x^{-1})&=x(\zeta_\phi\cdot\omega_\phi)(x^{-1})-(\zeta_\phi\cdot\omega_\phi)(1)=1;\\
\left(\zeta_\phi\cdot\omega^{(1)}_\phi\right)(yx^{-i-1})&=-\left(\phi^*_{yx^{-i-1}}\right)_1\\
\left(\zeta_\phi\cdot\omega^{(1)}_\phi\right)(x^{-1})&=0.
\end{align*}
By Lemma \ref{lemma 1 g=1}, and using that $\phi^*_{yx^{-g-1}}$ has degree $1$ in $\tau$, we deduce the following identities:
\begin{align*}
    (x\otimes1-1\otimes x)(\zeta_\phi\cdot\omega_\phi)=&\left(\sum_{i=0}^g \left(\phi^*_{yx^{-i}}-x\phi^*_{yx^{-i-1}}\right)_0\otimes x^i+1\otimes (y-Q(x))\right)(1\otimes\nu)\\
    =&\left(\left(\phi^*_{y}\right)_0\otimes 1+1\otimes (y-Q(x))\right)(1\otimes\nu)\\
    +&\left(\sum_{i=0}^{g-1} \left(\phi^*_{yx^{-i-1}}\right)_0\otimes x^{i+1}-\sum_{i=0}^{g-1} \left(x\phi^*_{yx^{-i-1}}\right)_0\otimes x^i\right)(1\otimes\nu)\\
    =&(1\otimes x-x\otimes1)\left(\sum_{i=0}^{g-1}\left(\phi^*_{yx^{-i-1}}\right)_0\otimes x^i\right)(1\otimes\nu)\\
    +&\left(y\otimes 1+1\otimes \left(y-Q(x)\right)\right)(1\otimes\nu)
    \end{align*}
    \begin{align*}
    f_\phi=&\frac{(x\otimes1-1\otimes x)\left(\zeta_\phi\cdot\omega^{(1)}_\phi\right)}{(x\otimes1-1\otimes x)(\zeta_\phi\cdot\omega_\phi)}\\
    =&\frac{(x\otimes1-1\otimes x)\left(-\sum_{i=0}^g \left(\phi^*_{yx^{-i-1}}\right)_1\otimes x^i\right)}{y\otimes 1+1\otimes (y-Q(x))-(x\otimes1-1\otimes x)\left(\sum_{i=0}^{g-1}\left(\phi^*_{yx^{-i-1}}\right)_0\otimes x^i\right)}.\tag*{\qedhere}
\end{align*}
\end{proof}

\subsection{Comparison with known results in the case of elliptic curves}

The computations can be directly compared to the results of Green and Papanikolas, who tackled the particular case of an elliptic curve in \cite{GP18}. They assumed $\phi$ to be normalized and the period lattice $\Lambda_\phi$ to be isomorphic to $A$, and they set:
\begin{align*}
    \phi_x=x+x_1\tau+\tau^2, && \phi_y=y+y_1\tau+y_2\tau^2+\tau^3.
\end{align*}
They proved the following identities (see \cite[Thm. 7.1, Eqs. 18,26,27]{GP18}):
\[f_\phi=\frac{1\otimes y-y\otimes 1-((y_2-x_1^q)\otimes1)(1\otimes x-x\otimes 1)}{1\otimes x-x^q\otimes1+(y_1-x_1(y_2-x_1^q))\otimes1};\]
\[\zeta_\phi\cdot\omega_\phi=\frac{(x^q-y_1+x_1(y_2-x_1^q))^q\otimes1-1\otimes x}{f_\phi}.\]
Let's compare these results with Theorem \ref{prop hyper}. First, we need to compute the coefficients $(\phi^*_{yx^{-1}})_0,(\phi^*_{yx^{-1}})_1,(\phi^*_{yx^{-2}})_1$. Starting from the definition of $\phi^*_x$ and $\phi^*_y$ we can explicitly compute the first $3$ terms of $\phi^*_{yx^{-1}}$ using the identity $\phi^*_x\phi^*_{yx^{-1}}=\phi^*_y$:
\begin{align*}
    \phi^*_x&=\tau^{-2}+x_1^{q^{-1}}\tau^{-1}+x\\
    \phi^*_y&=\tau^{-3}+y_2^{q^{-2}}\tau^{-2}+y_1^{q^{-1}}\tau^{-1}+y\\
    \phi^*_{yx^{-1}}&=\tau^{-1}+(y_2-x_1^q)+(y_1^q-x_1^qy_2^q-x^{q^2}+x_1^{q^2+q})\tau+\dots,
\end{align*}
hence $(\phi^*_{yx^{-1}})_0=y_2-x_1^q$ and $(\phi^*_{yx^{-1}})_1=y_1^q-x_1^qy_2^q-x^{q^2}+x_1^{q^2+q}$.
Since $\deg(yx^{-2})=-1$, and since $\phi$ is normalized, we have $\phi^*_{yx^{-2}}\in\tau+\C[\![\tau]\!]\tau^2$, hence $(\phi^*_{yx^{-2}})_1=1$. By Theorem \ref{prop hyper}, we have:
\[(\zeta_\phi\cdot\omega_\phi)f_\phi=(\zeta_\phi\cdot\omega^{(1)}_\phi)=-\sum_{i=0}^g \left(\phi^*_{yx^{-i-1}}\right)_1\otimes x^i=-((y_1^q-x_1^qy_2^q-x^{q^2}+x_1^{q^2+q})\otimes 1+1\otimes x),\]
which agrees with Green and Papanikolas' formula for $\zeta_\phi\cdot\omega_\phi$.

\begin{oss}
    A posteriori, since the computations do not take into account the $A$\=/module structure of $\Lambda_\phi$, it turns out that the formulas found by Green and Papanikolas hold without the assumption $\Lambda_\phi\cong A$.
\end{oss}

\subsection{Link with the Hartl--Juschka pairing for Drinfeld modules in genus \texorpdfstring{$0$}{0}}

In the article \cite{HJ20}, given an abelian and $A$\=/finite Anderson $A$\=/module $\underline{E}=(E,\phi)$, Hartl and Juschka define a perfect pairing of $A_\C$\=/modules between $N(\underline{E})$ and $\tau M(\underline{E})$, where $M(\underline{E})$ and $N(\underline{E})$ are respectively the $A$\=/motive and the dual $A$\=/motive associated to $\underline{E}$.

In this section we prove that Hartl and Juschka's perfect pairing coincides with the dot product defined in Lemma \ref{ev} when $\underline{E}$ is an $\F_q[\theta]$\=/Drinfeld module.

Let's first give the definition of $A$\=/motive and dual $A$\=/motive in the special case of a Drinfeld module. Let's denote by $A_\C[\tau]$ the noncommutative ring $\C[\tau]\otimes A$, and similarly for $A_\C[\tau^{-1}]$.

\begin{Def}\label{def (dual) motive}
    Let $\phi:A\to\C[\tau]$ be a Drinfeld module. 
    
    We define the $A$\=/motive $M(\phi)$ as the left $A_\C[\tau]$\=/module $\C[\tau]$ where for all $a\in A$, $h\in\C[\tau]$, $m\in M(\phi)$ we have $h\cdot m\coloneqq hm$ and $a\cdot m\coloneqq m\phi_a$.

    We define the dual $A$\=/motive $N(\phi)$ as the left $A_\C[\tau^{-1}]$\=/module $\C[\tau^{-1}]$ where for all $a\in A$, $h\in\C[\tau^{-1}]$, $m\in N(\phi)$ we have $h\cdot m\coloneqq hm$ and $a\cdot m\coloneqq m\phi^*_a$.  
\end{Def}

In the following proposition, we consider $\C\hat\otimes\Lambda_\phi$ as a left $\C\otimes A[\tau^{-1}]$\=/module, where $\tau^{-1}$ sends $x\in\C\hat\otimes\Lambda_\phi$ to $x^{(-1)}$, and we consider $\C\hat\otimes\Hom_A(\Lambda_\phi,\Omega)$ as a left $\C\otimes A[\tau]$\=/module where $\tau$ sends $x\in\C\hat\otimes\Hom_A(\Lambda_\phi,\Omega)$ to $x^{(1)}$.

\begin{prop}\label{prop motive embedding}
    The $\C[\tau^{-1}]$\=/linear morphism \[F:N(\phi)\cong\C[\tau^{-1}]\to\C\hat\otimes\Lambda_\phi\] sending $1$ to $\zeta_\phi$ is an injective morphism of $\C\otimes A[\tau^{-1}]$\=/modules.

    The $\C[\tau]$\=/linear morphism \[G:\tau M(\phi)\cong\tau\C[\tau]\to\C\hat\otimes\Hom_A(\Lambda_\phi,\Omega)\] sending $\tau$ to $\omega_\phi^{(1)}$ is an injective morphism of $\C\otimes A[\tau]$\=/modules.
\end{prop}
\begin{proof}
    Let's prove the $A$\=/linearity of $F$. Let $x:=\sum_i c_i\tau^{-i}\in N(\phi)$, so that its image is $F(x)=\sum_i (c_i\otimes 1)\zeta_\phi^{(-i)}$; for all $a\in A$, since $\zeta_\phi$ is a dual Anderson eigenvector we have:
    \[F(x\circ\phi_a^*)=(x\circ\phi_a^*)(F(1))=x(\phi_a^*(\zeta_\phi))=x(\zeta_\phi\cdot(1\otimes a))=(1\otimes a)\cdot x(\zeta_\phi)=(1\otimes a)F(x).\]
    To prove injectivity, let's fix $\sum c_i\tau^{-i}\in N(\phi)\setminus\{0\}$ and prove that $\sum c_i\zeta_\phi^{(-i)}$ is a nonzero element of $\C\hat\otimes\Lambda_\phi$. Let $N$ be the smallest index such that $c_N\neq0$. On one hand, for all $i>N$, \[\zeta_\phi^{(-i)}\cdot\omega_\phi^{(-N)}=\left(\zeta_\phi\cdot\omega_\phi^{(i-N)}\right)^{(-i)}\in\C\otimes\Omega\] by Theorem \ref{Teo rationality}; on the other hand, by Proposition \ref{cor 0 forms}, $\zeta_\phi^{(-N)}\cdot\omega_\phi^{(-N)}$ sends any $c\in\K$ with norm less than $1$ to $c^\frac{1}{q^n}$, as a function from ${\K}/{A}$ to $\C$. Since any form in $\C\otimes\Omega\subseteq\C\hat\otimes\Omega\cong\Hom_{\F_q}^\mathrm{cont}\left({\K}/{A},\C\right)$ has finite support, we deduce that $\zeta_\phi^{(-N)}\cdot\omega_\phi^{(-N)}\not\in\C\otimes\Omega$, hence \[\left(\sum_{i\geq0} c_i\zeta_\phi^{(-i)}\right)\cdot\omega_\phi^{(-N)}=c_N\left(\zeta_\phi^{(-N)}\cdot\omega_\phi^{(-N)}\right)+\sum_{i>N} c_i\left(\zeta_\phi^{(-i)}\cdot\omega_\phi^{(-N)}\right)\not\in\C\otimes\Omega;\]
    in particular, $\sum c_i\zeta_\phi^{(-i)}\neq0$.

    Let's prove the $A$\=/linearity of $G$. Let $y\in\tau M(\phi)$ and fix $x:=\sum_i c_i\tau^i\in\C[\tau]$ so that $x\tau=y$ and $G(y)=x(G(\tau))=\sum_i (c_i\otimes 1)\omega_\phi^{(i+1)}$; for all $a\in A$, since $\omega_\phi$ is an Anderson eigenvector we have:
    \[G(y\phi_a)=(x\tau\phi_a\tau^{-1})G(\tau)=(x\tau)(\phi_a(\omega_\phi))=(1\otimes a)x(\omega_\phi^{(1)})=(1\otimes a)G(y).\]
    To prove injectivity, let's fix $\sum_i c_i\tau^i\in\tau M(\phi)\setminus\{0\}$ and prove that $\sum c_i\omega_\phi^{(i)}$ is a nonzero element of $\C\hat\otimes\Hom_A(\Lambda_\phi,\Omega)$. Let $N$ be the smallest index such that $c_N\neq0$: as shown above, for all $i>N$, $\zeta_\phi^{(N)}\cdot\omega_\phi^{(i)}\in\C\otimes\Omega$, while $\zeta_\phi^{(N)}\cdot\omega_\phi^{(N)}\not\in\C\otimes\Omega$, so $\sum c_i\omega^{(i)}\neq0$.
\end{proof}

Let's state a version of Hartl and Juschka's theorem for Drinfeld modules.

\begin{teo}[{\cite[Thm. 5.13]{HJ20}}]
Let $\phi:A\to\C[\tau]$ be a Drinfeld module. There is a canonical $A_\C$\=/linear perfect pairing $HJ:N(\phi)\otimes_{A_\C}\tau M(\phi)\to\Omega_\C$.
\end{teo}

Hartl and Juschka leave as an open question the computation of the pairing $HJ$ in the general case, but they carry it out in the case $A=\F_q[\theta]$. In particular, they show the following.

\begin{prop}[{\cite[Ex. 5.16]{HJ20}}]\label{prop Hartl computation}
    Assume $A=\F_q[\theta]$ and let $\phi:A\to\C[\tau]$ be a Drinfeld module of rank $r$, with $\phi_\theta=\sum_i t_i\tau^i$. Let $\{\alpha_{i,j}\}_{0\leq i,j<r}\in \mathbb{C}_\infty^{r\times r}$ be the matrix with entries $\alpha_{i,j}:=-t_{i+j+1}^{q^{-i}}$, and let $\{\beta_{i,j}\}_{0\leq i,j<r}\in \mathbb{C}_\infty^{r\times r}$ be its inverse. Then for all $0\leq i,j<r$, the following identity holds: 
    \[HJ(\tau^{-j}\otimes\tau^{i+1})=\beta_{i,j}d\theta.\]
\end{prop}

We prove the following Theorem.

\begin{teo}\label{teo restricted pairing}
    Assume $A=\F_q[\theta]$ and let $\phi:A\to\C[\tau]$ be a Drinfeld module. The following identity holds in $\Omega_\C$ for all $i,j\geq0$:
    \[HJ(\tau^{-j}\otimes\tau^{i+1})=\zeta_\phi^{(-j)}\cdot\omega_\phi^{(i+1)}.\] 
\end{teo}
\begin{proof}
    By Proposition \ref{prop motive embedding}, we can identify $\tau M(\underline{E})=\Span_{\C}\{\omega_\phi^{(i+1)}\}_{i\geq0}$ and $N(\underline{E})=\Span_{\C}\{\zeta_\phi^{(-j)}\}_{j\geq0}$, hence we need to prove that the dot product coincides with the Hartl--Juschka pairing.

    If we call $r$ the rank of $\phi$, the set $\{\tau^{i+1}\}_{0\leq i<r}$ generates $\tau M(\underline{E})$ as an $A_\C$\=/module, and the set $\{\tau^{(-j)}\}_{0\leq j<r}$ generates $N(\underline{E})$ as an $A_\C$\=/module; since both the Hartl--Juschka pairing and the dot product are $A_\C$\=/linear, it suffices to prove the statement for all $0\leq i,j< r$. 
    
    By Proposition \ref{prop Hartl computation}, we need the following identity to hold in $\Omega_\C\subseteq\C\hat\otimes\Omega$ for all $0\leq i,j< r$:
    \[\sum_{i=0}^{r-1} (t_{k+i+1}\otimes1)^{q^{-k}}\left(\zeta_\phi^{(-j)}\cdot\omega_\phi^{(i+1)}\right)=-\delta_{k,j}d\theta.\] 
    
    If $k>j$, we have:
    \begin{align*}
        \sum_{i=0}^{r-1}(t_{k+i+1}\otimes1)^{q^{-k}}\left(\zeta_\phi^{(-j)}\cdot\omega_\phi^{(i+1)}\right)=&\sum_{i=r-k}^{r-1}(t_{k+i+1}\otimes1)^{q^{-k}}\left(\zeta_\phi^{(-j)}\cdot\omega_\phi^{(i+1)}\right)\\
        +&\sum_{i=0}^{r-1-k}(t_{k+i+1}\otimes1)^{q^{-k}}\left(\zeta_\phi^{(-j)}\cdot\omega_\phi^{(i+1)}\right)=0,
    \end{align*}
    where the first sum is $0$ because $t_l=0$ if $l>r$, and the second sum is $0$ because, by Proposition \ref{g=0 i=0}, \[\zeta_\phi^{(-j)}\cdot\omega_\phi^{(i+1)}=\left(\zeta_\phi\cdot\omega_\phi^{(i+j+1)}\right)^{(-j)}=0\text{ if }0<i+j+1<r,\]
    which is true because $i,j\geq0$ and $i+1\leq r-k<r-j$.

    Since $\omega_\phi$ is an Anderson eigenvector, the identity $\sum_l (t_l\otimes1)\omega_\phi^{(l)}=\omega_\phi(1\otimes\theta)$ holds, hence if $k\leq j$ we have:
    \begin{align*}
        &\left(\sum_{i=0}^{r-1}(t_{k+i+1}\otimes1)^{q^{-k}}\left(\zeta_\phi^{(-j)}\cdot\omega_\phi^{(i+1)}\right)\right)^{(k)}\\
        =&\zeta_\phi^{(k-j)}\cdot\sum_{i=k+1}^{r+k}(t_i\otimes1)\omega_\phi^{(i)}=\zeta_\phi^{(k-j)}\cdot\sum_{i=k+1}^{r}(t_i\otimes1)\omega_\phi^{(i)}\\
        =&(1\otimes\theta-\theta\otimes1)\zeta_\phi^{(k-j)}\cdot\omega_\phi-\zeta_\phi^{(k-j)}\cdot\sum_{i=1}^k (t_i\otimes1)\omega_\phi^{(i)}\\
        =&(1\otimes\theta-\theta\otimes1)\left(\zeta_\phi\cdot\omega_\phi^{(j-k)}\right)^{(k-j)}-\sum_{i=1}^k (t_i\otimes1)\left(\zeta_\phi\cdot\omega_\phi^{(i+j-k)}\right)^{(k-j)}.
    \end{align*}
    By Proposition \ref{g=0 i=0}, since $0<i+j-k\leq j<r$, the sum on the right hand side is $0$, while $(1\otimes\theta-\theta\otimes1)\left(\zeta_\phi\cdot\omega_\phi^{(j-k)}\right)^{(k-j)}$ is $0$ if $k<j$ and $-d\theta$ if $k=j$.
\end{proof}

\noindent{\footnotesize \textbf{Data availability} Data sharing not applicable to this article as no datasets were generated or analysed during the current study.}

\noindent{\footnotesize \textbf{Conflict of interest} The corresponding author states that there is no conflict of
interest.}
\bibliography{main}
\end{document}